\documentclass[a4paper,leqno]{amsart}
\usepackage{mathrsfs}
\usepackage{cite}
\usepackage{amssymb}
\usepackage{latexsym}
\usepackage{amsmath}

\newcommand{\rr}[1]{\mathbf R^{#1}}
\newcommand{\ep}{\varepsilon}
\newcommand{\fy}{\varphi}
\newcommand{\cdo}{\, \cdot \, }
\newcommand{\supp}{\operatorname{supp}}
\newcommand{\ON}{\operatorname{ON}}
\newcommand{\vrum}{\vspace{0.1cm}}

\newcommand{\scal}[2]{\left\langle #1,#2 \right\rangle}
\newcommand{\nm}[2]{\left\Vert #1  \right\Vert _{#2}}
\newcommand{\abp}[1]{\left\vert #1  \right\vert}
\newcommand{\sets}[2]{\left\{ \, #1\, ;\, #2\,  \right\} }
\newcommand{\eabs}[1]{\left\langle #1 \right\rangle}

\newcommand{\Bignm}[2]{\Bigl\Vert #1  \Bigr\Vert _{#2}}

\DeclareMathOperator{\Op}{\mathsf{Op}}

\numberwithin{equation}{section}          %Detta g?r att man f?r
\newtheorem{thm}{Theorem}
\numberwithin{thm}{section}
\newtheorem{prop}[thm]{Proposition}

\newtheorem{lemma}[thm]{Lemma}
\newtheorem*{tom}{\rubrik}
\newcommand{\rubrik}{}
\theoremstyle{definition}

\newtheorem{defn}[thm]{Definition}

\theoremstyle{remark}

\newtheorem{rem}[thm]{Remark}              %T o m hit  bara allm
\author{Ernesto Buzano}

\address{Department of Mathematics,
University of Torino, Italy}

\email{ernesto.buzano@unito.it}

\author{Joachim Toft}

\address{Department of Mathematics and Systems Engineering,
V{\"a}xj{\"o} University, Sweden}

\email{joachim.toft@vxu.se}

\title{\textbf {Schatten-von Neumann properties in the Weyl calculus}}

\begin{document}

\begin{abstract}
Let $\Op _t(a)$, for $t\in \mathbf R$, be the pseudo-differential operator
$$
f(x) \mapsto (2\pi )^{-n}\iint a((1-t)x+ty,\xi )f(y)e^{i\scal {x-y}\xi}\, dyd\xi
$$
and let $\mathscr I_p$ be the set of  Schatten-von Neumann operators of order $p\in
[1,\infty ]$ on $L^2$. We are especially concerned with the Weyl case
(i.{\,}e. when $t=1/2$). We prove that if $m$ and $g$ are
appropriate metrics and weight functions respectively, $h_g$ is the
Planck's function, $h_g^{k/2}m\in L^p$ for some $k\ge 0$ and $a\in S(m,g)$, then 
$\Op _t(a)\in \mathscr I_p$, iff $a\in L^p$. Consequently, if $0\le \delta
<\rho \le 1$ and $a\in S^r_{\rho ,\delta}$, then $\Op _t(a)$ is bounded
on $L^2$, iff $a\in L^\infty$.
\end{abstract}

\maketitle

%%%%%%%%%%%%%%%%%%%%%%%%%%%%%%%%%%%%
\section{Introduction}\label{sec0}
%%%%%%%%%%%%%%%%%%%%%%%%%%%%%%%%%%%%

\par

The aim of the paper is to continue the discussions in
\cite{BuNi, BuTo,To6} on general continuity and compactness properties
for pseudo-differential operators, especially for Weyl operators, with smooth symbols which belongs to certain H{\"o}rmander classes. We are especially focused on finding
necessary and sufficient conditions on particular symbols in order
for the corresponding pseudo-differential operators should be
Schatten-von Neumann operators of certain degrees. 

\par

If $V$ is a real vector space of finite dimension $n$, $V'$ its dual
space, $t\in \mathbf R$ is fixed and $a\in \mathscr S'(V\times V')$
(we use the same notation for the usual functions and distribution
spaces as in \cite{Ho3}), then the pseudo-differential operator $\Op
_t(a)$ of $a$ is a continuous linear map from $\mathscr S(V)$ to
$\mathscr S'(V)$ defined by
\begin{equation}
\Op_t(a)f(x) = (2\pi )^{-n}\iint _{V\times V'}
a((1-t)x+ty,\xi)f(y)e^{i\scal {x-y}\xi}\, dyd\xi. \label{t-op}
\end{equation}
(In the case when $a$ is not an integrable
function, $\Op _t(a)$ is interpreted as the operator with Schwartz
kernel equal to $(2\pi )^{-n/2}\mathscr F_2^{-1}a((1-t)x+ty,x-y)$,
where $\mathscr F_2U(x,\xi )$ denotes the partial Fourier transform
$\mathscr F$ on $U(x,y)$ with respect to the second variable. Here
$\mathscr F$ is the Fourier transform which takes the form
\begin{equation}\label{fourtrans}
\mathscr Ff(\xi )=\widehat f(\xi ) =(2\pi )^{-n/2}\int f(x)e^{-i\scal
x\xi}\, dx,
\end{equation}
when $f\in \mathscr S(V)$. See also Section 18.5 in \cite{Ho3}.) The operator $\Op _{1/2}(a)$ is  the Weyl operator of $a$, and is denoted by $\Op ^w(a)$. (See \eqref{t-op}$'$ in Section \ref{sec1}.)

\par

A family of symbol classes, which appears in several situations,
concerns $S^r_{\rho ,\delta }(\rr {2n})$, for $r,\rho ,\delta
\in \mathbf R$, which consists of all smooth functions $a$
on $\rr {2n}$ such that
$$
|\partial _x^\alpha \partial _\xi ^\beta a(x,\xi )|\le C_{\alpha
,\beta}\eabs \xi ^{r+|\alpha |\delta -|\beta |\rho }.
$$
Here $\eabs \xi=(1+|\xi|^2)^{1/2}$. By letting $s_{t,\infty}$ be the set
of all $a\in \mathscr S'$ such that the definition of $\Op _t(a)$
extends to a continuous operator on $L^2$, the following is a
consequence of Theorem 18.1.11 and the comments on page 94 in \cite{Ho3}:
\emph{Assume that $0\le \delta \le \rho \le 1$ and $\delta <1$. Then
$S^r_{\rho ,\delta}\subseteq s_{t,\infty}$ if and only if $r\le 0$.} The latter
equivalence can also be formulated as
\begin{equation}
S^r_{\rho ,\delta } \subseteq s_{t,\infty}\quad
\Longleftrightarrow \quad S^r_{\rho ,\delta } \subseteq L^\infty
.\label{Srest}
\end{equation}

\par

A similar property holds for any ``reasonable'' family of symbol
classes. This is a consequence of the investigations in \cite{BeF1,
BeF2, Ho1,  Ho3}. For example, in \cite{Ho1, Ho3},
H{\"o}rmander introduces a family of symbol classes, denoted by
$S(m,g)$, which is parameterized by the weight function $m$ and the
Riemannian metric $g$. (See Section \ref{sec1} for strict definition.)
By choosing $m$ and $g$ in appropriate ways, it
follows that most of those reasonable symbol classes can be obtained,
e.{\,}g. $S^{r}_{\rho ,\delta }$ is obtained in such way. If $m$ and
$g$ are appropriate, then \eqref{Srest} is generalized into:
\begin{equation}
S(m,g)\subseteq s_{t,\infty}\quad
\Longleftrightarrow \quad S(m,g) \subseteq L^\infty
.\tag*{(\ref{Srest})$'$}
\end{equation}
(Here we remark that important contributions for improving
the calculus on $S(m,g)$ can be found in \cite{Bn1,Bn2,BC,BoL}.
For example in \cite{Bn2}, Bony extends parts of the theory to a family
of symbol classes which contains any S(m,g) when $m$ and $g$
are appropriate.)

\par

In \cite{BuNi,To6}, the equivalence \eqref{Srest}$'$ is extended in
such way that it involves Schatten-von Neumann properties. More
precisely, let $s_{t,p}(V\times V')$ be the set of all $a\in \mathscr
S'(V\times V')$ such that $\Op _t(a)$ belongs to $\mathscr I_p$, the set
of Schatten-von Neumann operators of order $p\in [1,\infty ]$ on
$L^2(\rr n)$. (Cf. Section \ref{sec1} for a strict definition of
Schatten-von Neumann classes.) Then  in Theorem 1.1
in \cite{BuNi} equivalence \eqref{Srest}$'$ is generalized into
\begin{equation}
S(m,g)\subseteq s_{t,p}\quad
\Longleftrightarrow \quad S(m,g) \subseteq L^p
.\tag*{(\ref{Srest})$''$}
\end{equation}
provided certain extra conditions are imposed on $g$
comparing to \cite{Ho1,Ho2,Ho3}. In \cite{To6}, Theorem 1.1 in
\cite{BuNi} is improved, in the sense that the equivalence
\eqref{Srest}$''$ still holds without these extra  conditions
on $g$ (cf. Theorem 4.4 in \cite{To6}).

\par

Obviously, \eqref{Srest}$''$ completely characterizes the symbol
classes of the form $S(m,g)$ that are contained in
$s_{t,p}$. Consequently, a complete characterization of operator
classes of the form $\Op _t(S(m,g))$ to be contained in $\mathscr I_p$
follows from \eqref{Srest}$''$. On the other hand, \eqref{Srest}$''$
might give rather poor information about Schatten-von Neumann
properties for a \emph{particular} pseudo-differential operator $\Op
_t(a)$, when $a$ belongs to a fixed but arbitrary symbol class $S(m,g)$. For example, if
$a\in S(m,g)\nsubseteq L^p$, then \eqref{Srest}$''$ does not give any
information whether $\Op _t(a)$ belongs to $\mathscr I_p$ or not.

\par

In this context, Theorem 3.9 in \cite{Ho2} seems to be more adapted to
particular pseudo-differential operators with symbols in $S(m,g)$,
instead of whole classes of such operators. The theorem can be
formulated as:
\begin{equation}\label{indsymbA}
\text{Assume that $h_g^{N/2}m\in L^p$ holds for some $N\ge 0$ and $a\in S(m,g)$,}
\end{equation}
for $p=1$. Then
\begin{equation}\label{indsymb}
a\in L^p \qquad \Longrightarrow \qquad \Op _t(a)\in \mathscr
I_p,
\end{equation}
for $p=1$ and  and $t=1/2$.
Equivalently, if \eqref{indsymbA} holds for $p=1$, then
\begin{equation}\tag*{(\ref{indsymb})$'$}
a\in L^p \qquad \Longrightarrow \qquad a\in s_{t,p},
\end{equation}
for $p=1$ and  and $t=1/2$.
Theorem 3.9 in \cite{Ho2} is extended in \cite{To6}, where it is
proved that if \eqref{indsymbA} holds for some $p\in [1,\infty ]$, then \eqref{indsymb} and
\eqref{indsymb}$'$ hold for arbitrary $p$ and $t$. (Cf. Theorem
4.4$'$ and Remark 6.4 in \cite{To6}.)

\par

In Section \ref{sec2} in the present paper we prove that if
\eqref{indsymbA} holds, then \eqref{indsymb} and \eqref{indsymb}$'$
holds with the oposite implication. Consequently, if
\eqref{indsymbA} holds, then
\begin{equation}\label{indsymb2}
a\in L^p \qquad \Longleftrightarrow \qquad \Op _t(a)\in \mathscr
I_p.
\end{equation}
(See Theorem \ref{corthm12} and Theorem \ref{corthm13}.) Here
we note that a different proof of \eqref{indsymb2} in the case
$p=\infty$ can be found in \cite{BuTo}.

\par

In Section \ref{sec4} we also give some further remarks on embeddings
of the form \eqref{indsymb} in the case $p\in [1,2]$ and $t=1/2$ (the
Weyl case). More precisely, Theorem 3.9 in \cite {Ho2} was generalized
in Proposition 4.5$'$ in \cite{To6} as remarked at the
above. On the other hand, the proof of Theorem 3.9 in \cite{Ho2}
contains some techniques which are not
available in \cite {To6}. In Section \ref{sec4} we
combine these techniques with arguments in harmonic analysis to
prove some stressed estimates of the $s_{t,p}$ norm of compactly
supported elements in $C^N$. (See Lemmas
\ref{bernstein}--\ref{lemma3.8}, which
might be useful in other problems in the future as well.) Thereafter we
combine these estimates with arguments in the
proofs of Theorem 4.4$'$ and Proposition 4.5$'$ in \cite {To6}. These
investigations lead to Theorem \ref{thm3.9}, where slight different
sufficiency conditions on the symbols comparing to Theorem 4.4$'$ and
Proposition
4.5$'$ in \cite {To6} are obtained in order for the corresponding
pseudo-differential operators should be Schatten-von
Neumann operators of certain degrees. Roughly speaking, the main
differences between Proposition 4.5$'$ (or Theorem 4.4$'$) in \cite
{To6} and Theorem \ref{thm3.9} is that less regularity is imposed on
the symbols in Theorem \ref{thm3.9}, while weaker assumptions are
imposed on the parameterizing weight functions in Proposition 4.5$'$
in \cite {To6}.

\par

Finally, in Section \ref{sec3} we apply our results to symbol
classes, which are related to $S^r_{\rho ,\delta }$.

\par

%%%%%%%%%%%%%%%%%%%%%%%%%%%%%%%%%%%%
\section{Preliminaries}\label{sec1}
%%%%%%%%%%%%%%%%%%%%%%%%%%%%%%%%%%%%

\par

In this section we recall some well-known facts which are
needed. After a short review about integration over vector spaces, we
continue with discussing certain facts on symplectic vector
spaces. Thereafter we recall the definition of the symbol classes, and
discuss appropriate conditions for the Riemannian metrics and weight
functions which parameterize these classes.

\subsection{Integration on vector spaces}\rule{0pt}{0pt}

\medspace

In order to formulate our problems in a coordinate invariant way, we
 consider, as in \cite{To3, To5, To6},  
integration of \emph{densities} on a real vector space $V$ of finite
dimension $n$. A \emph{volume form on $V$} is a non-zero mapping
$\mu:\wedge^n V\setminus \{ 0\}\to\mathbf C$ which is positive
homogeneous of order one, i.{\,}e.\ such that
$\mu (t\omega )=|t| \mu (\omega )$, when $t\in \mathbf R\setminus \{
0\}$ and $\omega \in \wedge ^n(V)\setminus \{ 0\}$.
Since $\wedge^n V$ has dimension $1$, the volume form   $\mu$ is
completely determined by $\mu(e_1\wedge \cdots \wedge e_n)$, where
$e_1,\ldots,e_n$ is a basis of $V$.

\par

If we fix a \emph{volume form} $\mu$, it is possible to associate to
each  function $f:V\to \mathbf C$  a density $f\mu$ and define
\begin{equation}
\int_Vf\mu\,dx\equiv \idotsint_{\rr n} f(x_1 e_1+\cdots
+x_ne_n)\mu(e_1\wedge\cdots\wedge e_n) \,dx_1\cdots dx_n,
\label{eqn:6}
\end{equation}
where $e_1,\ldots,e_n$ is \emph{any} basis of $V$ and $x=\sum_{i=1}^n
x_i e_i$. In fact, it is easy to prove that the integral $\int_V
f\mu\,dx$ does not depend  on the choice of the basis $e_1,\ldots,e_n$
of $V$, even though it depends on the volume form $\mu$.

\par

If we  consider only bases $e_1,\ldots, e_n$ for $V$ such that
$$
\mu(e_1\wedge\cdots\wedge e_n)=1,
$$ 
\eqref{eqn:6} assumes the simpler form 
$$
\int_Vf\mu\,dx= \idotsint_{\rr n} f(x_1 e_1+\cdots
+x_ne_n)\,dx_1\cdots dx_n,
$$
and therefore we can omit $\mu$ in the left  hand side, i.{\,}e.
$$
\int_Vf\,dx= \idotsint_{\rr n} f(x_1 e_1+\cdots +x_ne_n)\,dx_1\cdots
dx_n,
$$

\par

Definition \eqref{eqn:6} allows to consider invariant $L^p(V)$
spaces. Since invariant definition of spaces of differentiable
functions like $C_0^\infty (V)$ and $\mathscr S(V)$ is not a problem,
we can also consider the dual spaces of distributions as $\mathscr
D'(V)$ and $\mathscr S'(V)$.

\par

If $f$ and $g$ belong to $\mathscr S(V)$, we consider the pairing
$$
\scal fg\equiv \int _Vfg \mu \,dx,
$$
which extends to the dual pairing between   $\mathscr S(V)$ and
 $\mathscr S'(V)$. 
  We also let 
$$
(f,g)=\scal f{\overline g}
$$
for admissible $f$ and $g$. The extension of $(\cdo ,\cdo )$ from
$\mathscr S(V)$ to $L^2(V)$ is
then the usual scalar product. 

\par

\subsection{Symplectic vector spaces}\label{SympVS}\rule{0pt}{0pt}

\medspace

Next we recall some facts about symplectic vector spaces. A real
vector space $W$ of finite dimension $2n$ is called
\emph{symplectic} if there exists  a
non-degenerate anti-symmetric bilinear form $\sigma$ on $W$, i.{\,}e. 
\begin{gather*}
\sigma (X,Y)=-\sigma (Y,X),\qquad\text{for all $X,Y\in W$,} 
\intertext{and}
\sigma (X,Y)=0,\quad \forall \ Y\in W\qquad \Longrightarrow \qquad X=0. 
\end{gather*}
The form $\sigma$ is called the \emph{symplectic form} of $W$.

\par

A basis $e_1,\dots ,e_n,\ep_1,\dots ,\ep_n$ for $W$ is called
\emph{symplectic} if it satisfies
$$
\sigma (e_j,e_k)=\sigma (\ep_j,\ep_k)=0,\quad 
\sigma (e_j,\ep_k)=-\delta _{jk},
$$
for $j,k=1,\dots ,n$. In some situations we use the notation
$e_{n+1},\dots ,e_{2n}$ for the vectors $\ep _1,\dots ,\ep _n$. Then,
with respect to this basis, $\sigma$ is given by
$$
\sigma (X,Y)=\sum_{j=1}^n (y_j\xi_j - x_j\eta_j),
$$
where
$$
X=\sum_{j=1}^n(x_j e_j+\xi_j \ep_j)\quad\text{and}\quad
Y=\sum_{j=1}^n(y_j e_j+\eta_j \ep_j).
$$ 
We refer to \cite {Ho3} for more facts about symplectic vector
spaces.

\par

In order to have invariant measure and integration on the symplectic
vector space $W$, we choose $\abp{\sigma^{\wedge n}}/n!$ as
\emph{symplectic volume form}. Since
$$
\sigma^{\wedge n}(e_1\wedge\cdots\wedge e_n\wedge \ep_1\wedge
\cdots\wedge \ep_n)=n!
$$
for a symplectic basis $e_1,\ldots,e_n,\ep_1,\ldots ,\ep_n$ (which we
sometimes abreviate as $e_1,\ldots ,\ep _n$), when we integrate  on
$W$, we can omit the symplectic volume form:
\begin{align*}
\int _W a(X)\, dX &=\iint_{\rr {2n}}a(x_1 e_1+\cdots
+x_ne_n+\xi_1\ep_1+\cdots +\xi _n\ep _n)
\, dx d\xi
\\[1ex]
&=\iint_{\rr {2n}}a(x_1 e_1+\cdots +\xi _n\ep _n)\, dx d\xi .
\end{align*}
With this choiche of volume form, the measure of subsets of $W$
coincides with the standard Lebesgue measure:
$$
\abp{U}=\int _W\chi_U\,dX=\iint_{\rr {2n}}\chi_U(x_1 e_1+\cdots +\xi
_n\ep _n)\, dxd\xi ,
$$
where $\chi_U$ is the characteristic function of $U\subseteq W$.

\medspace

The symplectic Fourier transform $\mathscr F_\sigma$ on $\mathscr
S(W)$ is defined by the formula
$$
\mathscr F_\sigma a(X)\equiv \pi ^{-n}\int_W a(Y)e^{2i\sigma (X,Y)}\,
dY,
$$
when $a\in \mathscr S(W)$. Then $\mathscr F_\sigma$ is a homeomorphism
on $\mathscr S(W)$ which extends to a homeomorphism on $\mathscr
S'(W)$, and to a unitary operator on $L^2(W)$. Moreover, 
$(\mathscr F_\sigma )^2$ is the identity operator. 
Also note that $\mathscr
F_\sigma$ is defined without any reference of symplectic
coordinates. 

\par

By straight-forward computations it follows that 
$$
\mathscr F_\sigma (a*b) = \pi ^n\mathscr F_\sigma a\, \mathscr
F_\sigma b,\quad
\mathscr F_\sigma (ab)= \pi ^{-n}\mathscr F_\sigma a *\mathscr
F_\sigma b,
$$
when $a\in \mathscr S'(W)$, $b\in \mathscr S(W)$, and $*$ denotes the
usual convolution. We refer to \cite {Fo, To1, To2, To3} for more
facts about the symplectic Fourier transform.

\par

Next we recall the definition of the Weyl quantization. 
Let $V$ be a real vector space of finite dimension $n$, $V'$ its dual
space and let $W=V\times V'$. The vector space $W$ has a natural
symplectic structure given by the symplectic form
\begin{equation}
\sigma(X,Y)=\scal y \xi- \scal x \eta,  \label{SympForm}
\end{equation}
where
$$
X=(x,\xi)\in V\times V',\quad Y=(y,\eta)\in V\times V',
$$
and $\scal \cdot \cdot $ is the duality pairing between $V$ and $V'$.

\par

\begin{rem}\label{splittedbasis}
Observe that when $W=V\times V'$, and $\sigma$ is defined as in
\eqref{SympForm}, then a symplectic basis for $W$ is given by
\emph{any}  basis $e_1,\ldots,e_n$ for $V\times \{0\}$ together with
its  dual basis $\ep_1,\ldots,\ep_n$ for $\{0\} \times V'$. We call
such a symplectic basis \emph{splitted}. Obviously, there are
symplectic bases which are not splitted.

\par

On the other hand, assume that $W$ is an $n$-dimensional symplectic vector space, $e_1,\dots ,e_n,\ep _1,\dots ,\ep _n$ is a fix symplectic basis, and $V$ and $V'$ are the vector spaces spanned by $e_1,\dots ,e_n$ and $\ep _1,\dots ,\ep _n$ respectively. Then $V'$ is the dual of $V$, with symplectic form as the dual form, and $W$ can be identified with $V\times V'$, in which the symplectic basis $e_1,\dots ,\ep _n$ is splitted.
\end{rem}

\par

The \emph{Weyl quantization $\Op ^w(a)$} of a symbol
$a\in \mathscr S'(W)$ is equal to $\Op _t(a)$ for $t=1/2$ (cf. the introduction). In particula, if $a\in \mathscr S(W)$ and $f\in \mathscr S(V)$, then 
\begin{equation}
\Op ^w(a)f(x) =
(2\pi )^{-n}\iint_{V\times V'} a\bigl((x+y)/2,\xi
\bigr)f(y)e^{i\scal {x-y}\xi}\, dyd\xi,  \tag*{(\ref {t-op})$'$}
\end{equation}
where $f\in \mathscr S(V)$ and the integration is performed with
respect to a splitted symplectic basis for  $W=V\times V'$.

\par

The definition of
$\Op ^w(a)$ extends to each $a\in \mathscr S'(W)$, giving a continuous
operator $\Op ^w(a):\mathscr S(V)\to\mathscr S'(V)$. (See \cite {Ho3,
To1, To2, To3}.) We also note that $\Op ^w(a)=\Op _{1/2}(a)$, when
$\Op _t(a)$ is given by \eqref{t-op}.

\par

\subsection{Operators and symbol classes}\rule{0pt}{0pt}

\medspace

We recall the definition of symbol classes which are
considered. (See \cite {Ho3}.) Assume that $a\in
C^N(W)$, $g$ is an arbitrary Riemannian metric on $W$, and that
$m>0$ is a measurable function on $W$. For each $k=0,\dots ,N$, let
\begin{equation}\label{e1.1}
|a|^g_k(X)=\sup |a^{(k)}(X;Y_1,\dots ,Y_k)|,
\end{equation}
where the supremum is taken over all $Y_1,\dots ,Y_k\in W$ such that
$g_X(Y_j)\le 1$ for  $j=1,\dots ,k$. Also set
\begin{equation}\label{e1.2}
\nm a{m,N}^g\equiv \sum _{k=0}^N\sup _{X\in W} \Bigl ( |a|^g_k(X)/m(X) \Bigr ),
\end{equation}
let $S_N{(m,g)}$ be the set of all $a\in
C^N(W)$ such that $\nm a{m,N}^g<\infty$, and let
\begin{equation*}
S{(m,g)} \equiv \bigcap _{N\ge 0}S_N{(m,g)}.
\end{equation*}

\par

Next we recall some properties for the
metric $g$ on $W$ (cf. \cite {To5, To6}). It follows from Section
18.6 in \cite {Ho3} that for each $X\in W$, there are symplectic
coordinates $Z=\sum_{j=1}^n(z_j e_j+\zeta_j \ep_j)$ which diagonalize
$g_X$, i.{\,}e. $g_X$ takes the form
\begin{equation}\label{e1.0}
g_X(Z) =\sum _{j=1}^n\lambda _j(X)(z_j^2+\zeta _j^2),
\end{equation}
where
\begin{equation}\label{e1.00}
\lambda _1(X)\ge \lambda _2(X)\ge \cdots \ge \lambda _n(X)>0,
\end{equation}
only depend on $g_X$
and are independent of the choice of symplectic coordinates which
diagonalize $g_X$.

\par

The \emph{dual metric} $g^\sigma$
and \emph{Planck's function} $h_g$ with respect to $g$ and the
symplectic form $\sigma$ are defined by
$$
g^\sigma _X(Z)\equiv \sup _{Y\neq 0}\frac {\displaystyle{\sigma
(Y,Z)^2}}{\displaystyle{g_X(Y)}} \quad \text{and}\quad h_g(X)
=\sup _{Z\neq 0} \Big ( \frac {\displaystyle
{g_X(Z)}}{\displaystyle {g_X^\sigma (Z)}}\Big )^{1/2}
$$
respectively. It follows that if \eqref {e1.0}
and \eqref {e1.00} are fulfilled, then $h_g(X)=\lambda _1(X)$ and
\begin{equation}\tag*{(\ref{e1.0})$'$}
g^\sigma _X(Z) =\sum _{j=1}^n\lambda _j(X)^{-1}(z_j^2+\zeta_j^2).
\end{equation}
In most of the applications we have that $h_g(X)\le 1$ everywhere,
i.{\,}e. the \emph{uncertainly principle} holds.

\par

The metric $g$ is called \emph{symplectic} if $g_X=g^\sigma _X$ for
every $X\in W$. It follows that $g$ is symplectic if and only if
$\lambda _1(X)=\cdots =\lambda _n(X)=1$ in \eqref {e1.0}.

\par

We recall that parallel to $g$ and $g^\sigma _{\phantom X}\!$, there is
also a canonical way to assign a corresponding symplectic metric
$g^{{}_0} _{\phantom X}\!$.  (See e.{\,}g. \cite {To6}.) More precisely, let
$Mg=(g+g^\sigma)/2$ and define
$$
g^{{}_0} _{X}=\lim_{k\to\infty}M^k g.
$$
Then $g^{{}_0}$ is a symplectic metric, defined in a symplectically
invariant way and if $Z=\sum_{j=1}^n(z_j e_i+\zeta_j \ep_j)$ are
symplectic coordinates  such that \eqref{e1.0} is fulfilled, then
$$
g^{{}_0} _X(Z) =\sum _{j=1}^n(z_j^2+\zeta_j^2).
$$

\par

The Riemannian metric $g$ on $W$ is called \emph{slowly varying}
if there are positive constants $c$ and $C$ such that
\begin{equation}\label{slowly}
g_X(Y-X)\le c\quad \Longrightarrow  \quad C^{-1}g_Y \le g_X\le Cg_Y.
\end{equation}
More generally, assume that $g$ and $G$ are Riemannian metrics on
$W$. Then $G$ is called \emph{$g$-continuous}, if there are
positive constants $c$ and $C$ such that
\begin{equation}\tag*{(\ref{slowly})$'$}
g_X(Y-X)\le c \quad \Longrightarrow
\quad C^{-1}G_Y \le G_X\le CG_Y.
\end{equation}
By duality it follows that $g$ is slowly varying if and only if
$g^\sigma$ is $g$-continuous, and that \eqref{slowly} is equivalent
to \eqref{slowly}$'$, when $G=g^\sigma$.

\par

A positive function $m$ on $W$ is called \emph{$g$-continuous} if
there are constants $c$ and $C$ such that
\begin{equation}\label{gcont}
g_X(Y-X)\le c \quad \Longrightarrow
\quad C^{-1}m(Y) \le m(X)\le Cm(Y).
\end{equation}

\par

We observe that if $g$ is slowly varying, $N\ge 0$ is an integer and
$m$ is $g$-continuous, then $S_N{(m,g)}$ is a Banach space when the
topology is defined by the norm \eqref{e1.2}. Moreover, $S{(m,g)}$ is
a Frech\'et space under the topology defined by the norms
\eqref{e1.2} for all $N\ge 0$.

\par

The Riemannian metric $g$ on $W$ is called
\emph{$\sigma$-temperate}, if there is a constant $C>0$ and an integer
$N\ge 0$ such that
\begin{equation}\label{e1.3}
g_Y(Z) \le  Cg_X(Z)(1+g^\sigma _Y(X-Y))^N,\quad\text{for all $X,Y,Z\in
W$.}
\end{equation}
We observe that if \eqref{e1.3} holds, then \eqref{e1.3} still holds
after the term $g^\sigma _Y(X-Y)$ is replaced by $g^\sigma
_X(X-Y)$, provided the constants $C$ and $N$ have been replaced by
larger ones if necessary. (See also \cite {Ho3}.)

\par

More generally, if $g$ and $G$ are Riemannian metrics on $W$,
then $G$ is called \emph{$(\sigma,g)$-temperate}, if there is a
constant $C$ and an integer $N\ge 0$ such that
\begin{equation}\tag*{(\ref{e1.3})$'$}
\begin{cases}
G_X(Z) \le  CG_Y(Z)(1+g^\sigma _X(X-Y))^N,
\\[1ex]
G_X(Z) \le  CG_Y(Z)(1+g^\sigma _Y(X-Y))^N,\quad \text{for all}\
X,Y,Z\in W.
\end{cases}
\end{equation}
By duality it follows that $G$ is $(\sigma,g)$-temperate, if and only
if $G^\sigma$ is $(\sigma,g)$-temperate. In particular, $g$ is
$\sigma$-temperate, if and only if $g^\sigma$ is
$(\sigma,g)$-temperate. We also note that if $g$ is
$\sigma$-temperate and one of the inequalities in \eqref{e1.3}$'$
holds, then $G$ is $(\sigma,g)$-temperate.

\par

The weight function $m$ is called $(\sigma,g)$-temperate if
\eqref{e1.3}$'$ holds after $G_X(Z)$ and $G_Y(Z)$ have been replaced
by $m(X)$ and $m(Y)$ respectively.

\medspace

In the following proposition we give examples on important functions
related to the slowly varying metric $g$ and which are symplectically
invariantly defined. Here we set
\begin{equation}\label{varthetadef}
\Lambda _g(X)=\lambda _1(X)\cdots \lambda _n(X),
\end{equation}
when $g_X$ is given by \eqref{e1.0}.

\par

\begin{prop}\label{massinv}
Assume that $g$ is a Riemannian metric on $W$, and that $X\in W$ is
fixed. Also assume that the symplectic coordinates are chosen such
that \eqref{e1.0} holds. Then the following
are true:
\begin{enumerate}
\item $\lambda _j$ for $1\le j\le n$ and $\Lambda _g$ are symplectically
invariantly defined;

\vrum

\item if in addition $g$ is slowly varying, then $\lambda _j$ for
$1\le j\le n$ and $\Lambda _g$ are $g$-continuous;

\vrum

\item if in addition $g$ is $\sigma$-temperate, then $\lambda _j$ for
$1\le j\le n$ and $\Lambda _g$ are $(\sigma,g)$-temperate.
\end{enumerate}
\end{prop}

\par

\begin{proof}
The assertion follows immediately from the fact that
$$
\lambda _j(X)=\inf_{W_j} \Bigl (\sup _{Y\in W_j\setminus 0}\Big ( \frac
{g_X(Y)}{g_X^\sigma (Y)}\Big )^{1/2}\Bigr ),
$$
where the infimum is taken over all symplectic subspaces $W_j$ of
$W$ of dimension $2(n-j+1)$.
\end{proof}

\par

We note that an alternative proof of (1) in Proposition \ref{massinv}
can be found in Section 18.5 in \cite {Ho3}.

\medspace

The following definition is motivated by the general theory of Weyl
calculus. (See Section 18.4--18.6 in \cite {Ho3}.)

\par

\begin{defn}\label{feasible}
Assume that $g$ is a Riemannian metric on $W$. Then
$g$ is called
\begin{enumerate}
\item[(i)] \emph{feasible} if $g$ is slowly varying and $h_g\le 1$
everywhere;

\vrum

\item [(ii)] \emph{strongly feasible} if $g$ is  feasible
and $\sigma$-temperate.
\end{enumerate}
\end{defn}

\par

Note that feasible and strongly feasible metrics are not standard
terminology. In the literature it is common to use the term
``H{\"o}rmander metric'' or ``admissible metric'' instead of
``strongly feasible'' for metrics which satisfy (ii) in Definition
\ref{feasible}. (See \cite {Bn1, Bn2, BC, BoL, BuNi}.) An important
reason for us to follow \cite {To5, To6} concerning this terminology
is that we permit metrics which are not admissible in the sense of
\cite{Bn1, Bn2, BC, BoL, BuNi}, and that we prefer similar names for
metrics which satisfy (i) or (ii) in Definition \ref{feasible}.

\par

\begin{rem}\label{feasrem}
We note that if $g$ is strongly feasible, then $g^{{}_0}$
is strongly feasible, and $g$ and $h_g^{-s}g$ are
$(\sigma,g^{{}_0})$-temperate when $0\le s \le
1$ (cf. \cite{To6}). In particular, $h_g^{-s }g$ is strongly
feasible which is also an immediate consequence of Proposition 18.5.6
in \cite {Ho3}.
\end{rem}

\par

\begin{rem}\label{remfeasible}
Assume that $g$ is slowly varying on $W$ and let $c$ be the same as in
\eqref{slowly}. Then it follows from Theorem 1.4.10 in \cite {Ho3}
that there is a constant $\ep >0$, an integer $N\ge
0$ and a countable sequence $\{X_j\} _{j\in \mathbf N}$ in $W$ such
that the following is true:
\begin{enumerate}
\item there is a positive number $\ep$  such that
$g_{X_j}(X_j-X_k)\ge \ep$ for every $j,k\in \mathbf N$ such that
$j\neq k$;

\vrum

\item $W =\bigcup _{j\in \mathbf N}U_j$, where $U_j$ is the
$g_{X_j}$-ball $\sets {X}{g_{X_j}(X-X_j)< c}$;

\vrum

\item the intersection of more than $N$ balls $U_j$ is empty.
\end{enumerate}
\end{rem}

\par

\begin{rem}\label{remsetfeasible}
It follows from Section 1.4 and Section 18.4 in \cite{Ho3} that if $g$
is a slowly varying metric on $W$, and (1)--(3)
in Remark \ref{remfeasible} holds, then there is a
sequence $\{ \fy _j\}_{j\in \mathbf N}$ in $C^\infty _0(W)$ such that
the following is true:
\begin{enumerate}
\item $0\le \fy _j\in C_0^\infty (U_j)$ for every $j\in \mathbf N$;

\vrum

\item $\sup _{j\in \mathbf N}\nm {\fy _j}{1,N}^{g_{X_j}}<\infty$ for every
integer $N\ge 0$ (i.{\,}e. $\{ \fy _j\}_{j\in \mathbf N}$ is a bounded
sequence in $S(1,g)$);

\vrum

\item $\sum _{j\in \mathbf N}\fy _j=1$ on $W$.
\end{enumerate}
\end{rem}

\subsection{Schatten-von Neumann operators}\rule{0pt}{0pt}

\medspace

Next we recall some facts about Schatten-von Neumann operators.
(see \cite {Si}.) Let $\ON _0(V)$ be the set of all finite
orthonormal sequences $\{f_j \}_{j\in J}$ in $L^2(V)$ such that
$f_j\in \mathscr S(V)$ for every $j\in J$.
Then the linear operator $T$ from $\mathscr S(V)$ to $\mathscr S'(V)$
is called a Schatten-von Neumann operator of order $p\in [1,\infty ]$
(on $L^2(V)$), if
\begin{equation}\label{schattennorm}
\nm T{\mathscr I_p}\equiv \sup \Big (\sum _{j\in J}|(Tf_j,g_j)|^p\Big
)^{1/p}<\infty,
\end{equation}
 where the supremum is taken over all sequences $\{ f_j\}
_{j\in J}$ and $\{ g_j\} _{j\in J}$ in $\ON _0(V)$. 
The set of Schatten-von Neumann operators of order $p$ is denoted by
$\mathscr I_p$. Then $\mathscr I_p$ is a Banach space under the norm
$\nm \cdo {\mathscr I_p}$, and
$\mathscr I_1$, $\mathscr I_2$ and $\mathscr I_\infty$ are the spaces of
trace-class, Hilbert-Schmidt, and continuous
operators on $L^2(V)$ respectively. Moreover, $\mathscr I_p$ increases
with $p$, $\nm {\cdot} {\mathscr I_p}$ decreases with $p$, and if
$T\in \mathscr I_p$ for $p<\infty$, then $T$ is compact on
$L^2(V)$. We refer to \cite {Si} for more facts about Schatten-von
Neumann spaces.

\medspace

For each $p\in [1,\infty ]$ and $t\in \mathbf R$, we let $s^{t,p}_p(W)$ be the set of all $a\in
\mathscr S'(W)$ such that $\Op _t(a) \in \mathscr I_p$. We also let
$s_{t,\sharp} (W)$ be the subspace of $s_{t,\infty} (W)$ consisting of all
$a$ such that $\Op _t(a)$ is compact on $L^2(V)$. The spaces $s_{t,p}(W)$
and $s_{t,\sharp} (W)$ are equipped by the norms $\nm a{s_{t,p}}\equiv \nm
{\Op _t(a)}{\mathscr I_p}$ and $\nm {\cdot}{s_{t,\infty}}$ respectively.
It follows that the map $a\mapsto \Op _t(a)$ is an isometric homeomorphism
from $s_{t,p}(W)$ to $\mathscr I_p$, for every $p\in
[1,\infty ]$ (see \cite {To1, To2, To3}). Since the Weyl case is particularily interesting we also use the notation $s^w_p$ and $s^w_\sharp$ instead of $s_{t,p}$ and $s_{t,\sharp}$ when $t=1/2$.

\par

In the following propositions, we recall some facts for the
$s^w_p$-spaces. The proofs are omitted since the results are
restatements of certain results in \cite {To1, To2, To3}. Here
and in what follows, $p'\in [1,\infty ]$ denotes the conjugate
exponent of $p\in [1,\infty ]$, i.{\,}e. $1/p+1/p'=1$. 
We also use the notation $L^\infty _0(W)$ for the set of all $a\in
L^\infty (W)$ such that
$$
\lim _{R\to \infty}\Big (\underset {|X|\ge R} {\operatorname
{ess{\,}sup}} |a(X)|\Big ) =0,
$$
where $|\cdot|$ is any  euclidean norm on $W$. We refer to \cite {To4} for more
facts about the $s_{t,p}$ spaces for general $t\in \mathbf R$. 

\par

\begin{prop}\label{p1.57A}
Assume that $p,p_1,p_2\in [1,\infty]$ are such that $p_1 \le
p_2<\infty$. Then $s^w_p(W)$ and $s^w_\sharp(W)$ are Banach spaces
with continuos embeddings
$$
\mathscr S(W)\hookrightarrow s^w_{p_1}(W)\hookrightarrow
s^w_{p_2}(W)\hookrightarrow s^w_\sharp (W)\hookrightarrow s^w_\infty
(W)\hookrightarrow \mathscr S'(W).
$$
Moreover, $s_2^w(W)=L^2(W)$. 

\par

If
$a\in \mathscr S'(W)$ and $T$ is an affine symplectic map, then
$\mathscr F_\sigma$ and the pullback $T^*$ are homeomorphisms on
$s_p^w(W)$ and on $s_\sharp ^w(W)$, and
\begin{alignat*}{2}
&\nm a{s^w_{p_2}} \le \nm a{s^w_{p_1}}, &\quad &
\nm a{s^w_p}=\nm {T^*a}{s^w_p}=\nm {\mathscr F_\sigma a}{s^w_p},
\\[1ex]
&\nm a{L^\infty} \le 2^{n}\nm a{s^w_1}, &\quad & \nm a{s^w_2} =
(2\pi )^{-n/2}\nm a{L^2} \text .
\end{alignat*}
\end{prop}

\par

\begin{prop}\label{p1.57B}
Assume that $p\in [1,\infty]$. Then the following is true:
\begin{enumerate}
\item the bilinear form $\scal \cdo \cdo $  on 
$\mathscr S(W)$ and the $L^2$-form $(\cdo ,\cdo )$ on 
$\mathscr S(W)$ extend uniquely to the duality between $s^w_p(W)$ and
$s^w_{p'}(W)$, and for every $a\in s^w_p(W)$ and $b\in s^w_{p'}(W)$ it
holds
\begin{gather*}
\abp {\scal a b}\le \nm a{s^w_p}\nm b{s^w_{p'}},\quad \abp {(a,b)}\le
\nm a{s^w_p}\nm b{s^w_{p'}}
\intertext{and}
\nm a{s^w_p}=\sup \abp {\scal a c}=\sup \abp {(a,c)}
\end{gather*}
where the supremums are taken over all $c\in s^w_{p'}(W)$ such
that $\nm c{s^w_{p'}}\le 1$;

\vrum

\item if $p<\infty$, then the dual space for $s^w_p$ can be identified
with $s^w_{p'}$ through the form $\scal \cdo \cdo $ or $(\cdo ,\cdo )$.
\end{enumerate}
\end{prop}

\par

In what follows we let $B_r(X)$ denote the open ball
with center at $X\in W$ and radius $r$, provided there is no
confusion about the euclidean structure in $W$. For future references
we also set $B(X)=B_1(X)$.

\begin{prop}\label{p1.57C}
Assume that $p\in [1,\infty ]$ and $r\in (1,\infty)$. Then 
$$
s_p^w(W)\cap \mathscr E'(W)
=\mathscr F_\sigma \bigl(L^p(W)\bigr)\cap \mathscr E'(W),
$$
 and for some
constant $C$ which only depends on $r$ and $n$ it holds
\begin{equation}\label{sova}
C^{-1}\nm {\mathscr F_\sigma  a}{L^p}\le \nm a{s_p^w} \le C\nm
{\mathscr F_\sigma  a}{L^p},
\end{equation}
for all $a\in \mathscr E'(B_r(0))$. Here the open ball $B_r(0)$ is
taken with respect to any euclidean metric.
\end{prop}

\par

The next proposition concerns interpolation properties. Here and in
what follows we use similar notations as in \cite {BeL} concerning
interpolation spaces.

\par

\begin{prop}\label{interpolprop}
Assume that $p,p_1,p_2\in [1,\infty ]$ and $0\le \theta \le 1$ such
that $1/p=(1-\theta )/p_1+\theta /p_2$. Then the
(complex) interpolation space $(s^w_{p_1},s^w_{p_2})_{[\theta ]}$ is
equal to $s^w_{p}$ with equality in norms.
\end{prop}

\par

%%%%%%%%%%%%%%%%%%%%%%%%%%%%%%%%%%%%%%%%%%%%%%%%%%%%%%%%%%%%%%%%%%%%%
\section{Necessary and sufficient conditions for symbols to define 
Schatten-von Neumann operators}\label{sec2}
%%%%%%%%%%%%%%%%%%%%%%%%%%%%%%%%%%%%%%%%%%%%%%%%%%%%%%%%%%%%%%%%%%%%%

\par

\subsection{Necessary and sufficient conditions for symbols in the
Weyl calculus}\label{subsec21}\rule{0pt}{0pt}

\medspace

In this subsection we continue the discussion from \cite{BuNi, To6}
concerning Schatten-von Neumann properties for pseudo-differential
operators. 
We discuss necessity for symbols in $S(m,g)$ in order
to the corresponding Weyl operators should be Schatten-von Neumann
operators of certain degrees. We essentially prove that
the sufficiency results in Section 6 in \cite {To6} are to some extent
also necessary. More precisely we have the following result.

\par

\begin{thm}\label{corthm12}
Assume that $p\in [1,\infty ]$, $g$ is strongly feasible,
$m$ is $g$-continuous, and that $a\in S(m,g)$. Then the following is
true:
\begin{enumerate}
\item if $h_g^{N/2}m\in L^p(W)$ for some $N\ge 0$, then $a\in
s_p^w(W)$ if and only if $a\in L^p(W)$;

\vrum

\item if $h_g^{N/2}m\in L^\infty _0(W)$ for some $N\ge 0$, then $a\in
s_\sharp ^w(W)$ if and only if $a\in L^\infty _0(W)$.
\end{enumerate}
\end{thm}

\par

Using completely different techniques, Theorem \ref{corthm12} has
already been proved in \cite{BuTo} when $p=\infty$ and $m$ is
$(\sigma,g)$-temperate. Here we prove it as an immediate consequence
of Proposition \ref{thm1} and Proposition \ref{thm2} below. The proof
of Proposition \ref{thm1} is omitted since the result is the same as
Proposition 4.5$'$ in \cite {To6}.

\par

Here and in what follows we set
\begin{equation} \label{k_p}
\kappa _p=\begin{cases}
2[2n(1/p-1/2)]+1,&\text{for  $p\in [1,2)$,}
\\
0,&\text{for  $p\in[2,\infty]$,}
\end{cases}
\end{equation}
where $[x]$ denotes the integer part of the real number $x$.

\par

\begin{prop}\label{thm1}
Assume that $g$ is slowly varying when $p\in[1,2]$ and strongly
feasible when $p\in ]2,\infty]$. Then the following is true:
\begin{enumerate}
\item if 
$h_g^{N/2}m\in L^p(W)$ for some $N\ge \kappa _p$, then we have
\begin{equation}
S_N(m,g)\cap L^p(W)\subseteq s_p^w(W),\label{eqn:1}
\end{equation}
and
\begin{equation}
\nm{a}{s^w_p}\le C\Bigl ( \nm{a}{L^p}+\nm{a}{m,N}^g\nm{h^{N/2}_g
m}{L^p}\Bigr),  \label{eqn:2}
\end{equation}
for some  constant $C$ which is independent of $a$;

\vrum

\item if 
$h_g^{N/2}m\in L^\infty _0(W)$ for some $N\ge 0$, then 
\begin{equation}
S_N(m,g)\cap L^\infty _0(W)\subseteq s_\sharp ^w(W).\label{eqn:3}
\end{equation}
\end{enumerate}
\end{prop}

\par

The next result is the needed converse of Proposition \ref{thm1}.

\par

\begin{prop}\label{thm2}
Assume that $g$ is strongly feasible  when $p\in [1,2)$ and feasible
when $p\in [2,\infty]$, and that $m$ is
$g$-continuous. Then the following is true:
\begin{enumerate}
\item if 
$h_g^{N/2}m\in L^p(W)$ for some $N\ge 0$, then 
\begin{equation}\label{eqn:4}
S(m,g)\cap s_p^w(W) \subseteq L^p(W)\text ;  
\end{equation}

\vrum

\item if  $h_g^{N/2}m\in L^\infty _0(W)$ for some $N\ge 0$, then
\begin{equation}
S(m,g)\cap s_\sharp ^w(W) \subseteq L^\infty _0(W).\label{eqn:5}
\end{equation}
\end{enumerate}
\end{prop}

\par

We need some preparations for the proof and start with the following
lemma, which is essentially the same as Lemma 3.1 of \cite{Sh}.

\par

\begin{lemma}\label{fderest}
Assume that $f\in C^2([0,r])$. Then
$$
|f'(0)|\le 4(r^{-1}+1)\Big ( \max _{t\in [0,r]}|f(t)|+\max _{t\in
[0,r]}|f''(t)|\Big ).
$$
\end{lemma}

\par

\begin{proof}
We may assume $f'(0)\ne 0$.  Set
$$
M_j=\max_{t\in [0,r]}|f^{(j)}(t)|,\quad\text{for $j=0,2$.}
$$
By the mean value theorem we have
$|f'(t)-f'(0)|\le M_2t$, for all $t\in[0,r]$. Then
$$
|f'(0)|/2\le |f'(t)|\qquad \text{for} \quad 2M_2t\le |f'(0)|\quad
\text{and}\quad 0\le t\le r.
$$

\par

Let $\delta=\min \{r,\,|f'(0)|/(2M_2)\}$.
Then using the mean value theorem again it follows that
$f(\delta)-f(0)=f'(s)\delta$ for some $s\in [0,\delta]$. This gives
$$
|f'(0)|/2\le |f'(s)|\le |f(\delta)-f(0)|/\delta\le 2M_0/\delta .
$$
Then either
$$
|f'(0)|/2\le 2M_0/r\quad \text{or}\quad |f'(0)|/2\le
4M_0M_2/|f'(0)|.
$$
The result now follows by combining these inequalities.
\end{proof}

\par

\begin{lemma}\label{lemma:1}
Assume that $g$ is slowly varying, $N\in \mathbf N$, and consider  the
open ball
$$
U_X=\sets {Y\in W}{g_X(Y-X)<c},
$$
where $c$ is the same as in \eqref{slowly}.  Then there exists a positive
constant $C_0$, depending only
on $N$, $n$ and the constants in \eqref{slowly}  such that
$$
\sup_{k\le N}\sup_{Y\in \overline U_X}|a|^g_k(Y)\le
C_0\Bigl(\sup_{Y\in \overline U_X}|a(Y)|+\sup_{Y\in \overline
U_X}|a|^g_N(Y)\Bigr),
$$
for all $X\in W$ and all  $a\in C^N(W)$.
\end{lemma}

\par

\begin{proof}
By induction we may assume $N=2$. Let $X_0\in \overline U_X$ be
fixed. We shall find an appropriate basis $e_1,\ldots,e_{2n}$,
orthonormal  with respect to $g_X$, and such that
\begin{equation}\label{Zcond}
X_0+t e_j\in U_X,
\qquad \text{for }\quad 0\le t<\sqrt{c/2n}.
\end{equation}
Let us  first show that it is always possible to find $e_1,\dots
,e_{2n}$ such that \eqref{Zcond} is fulfilled. Since this is obviously
true for $X_0=X$, we may assume $X_0\ne X$. Let
$$
\tilde g_X(Y,Z)=\bigl(g_X(Y+Z)-g_X(Y-Z)\bigr)/4
$$ 
be the  polarization of $g$ and  choose the basis $e_1,\ldots, e_{2n}$
such that
$$
\tilde g_X(X_0-X,e_j)=-\sqrt{g_X(X_0-X)/(2 n)},\qquad\text{for
$j=1,\ldots,2n$.}
$$
This is possible  if we choose $e_1,\dots ,e_{2n}$ in such way that
$X_0-X=-t_0(e_1+\cdots +e_{2n})$ for some $t_0>0$. Then we have
\begin{multline*}
g_X(X_0+t e_j-X)=g_X(X_0-X)-2t \sqrt{g_X(X_0-X)/(2 n)}+t^2
\\[1ex]
=(t-\sqrt{g_X(X_0-X)/(2 n)})^2 + (1-(2n)^{-1})g_X(X_0-X) < c,
\end{multline*}
since it follows from the assumptions that
$$
-\sqrt {c/(2n)} \le t-\sqrt{g_X(X_0-X)/(2 n)} \le \sqrt {c/(2n)}\qquad
\text{and}\qquad g_X(X_0-X)<c.
$$
Since $\tilde g_X(Z,e_j)\le1$, for  $g_X(Z)= 1$,   we have
$$
|a|^g_1(X_0)\le C  \sup_{\scriptscriptstyle g_X(Z)=1}
\left|a^{(1)}(X_0;Z)\right|\le C \sum_{j=1}^{2n}|a^{(1)}(X_0;e_j)|,
$$
where $C$ is the same as in \eqref{slowly}.

\par

If we let $f(t)=a(Z+t e_i)$ and $r=\sqrt {c/2n}$, Lemma \ref{fderest}
shows that
\begin{equation}
|a^{(1)}(X_0;e_i)|\le 4(\sqrt {2n/c}+1)\Bigl ( \max_{Y\in\overline
U_X}|a(Y)|+\sum_{j=1}^{2n}\max_{Y\in \overline
U_X}|a^{(2)}(Y;e_j,e_j)|\Bigr ), \label{eqn:9}
\end{equation}
for $i=1,\ldots,2n$. The result now follows from \eqref{eqn:9} and
$$
\sum_{j=1}^{2n}\max _{Y\in \overline U_X}|a^{(2)}a(Y;e_j,e_j)|
\le 2n\max_{Y\in \overline U_X}|a|^g_2(Y),
$$
which completes the proof.
\end{proof}

\par

\begin{proof}[Proof of Proposition \ref{thm2}.]
We only prove (1)  in the case $p<\infty$. The case $p=\infty$
and assertion (2) follow by similar arguments and are left for the
reader.

\par

Assume that $a\in S(m,g)$, $a\notin L^p(W)$ and set
$G_X=h_g(X)^{-1/2}g_X$. Furthermore, since $h_g\le 1$, it follows from
Remark~\ref{feasrem} that $G$ is feasible (strongly feasible) when $g$
is feasible (strongly feasible), and that the hypothesis still holds after $N$
has been replaced by a larger number. Hence, \eqref{eqn:10} and
Proposition~\ref{thm1} give that $h^{N/2}_G m^{p-1}\in L^{p'}$ and
\begin{equation}\tag*{(\ref{eqn:2})$'$}
\nm{b}{s^w_{p'}}\le C\Bigl (
\nm{b}{L^{p'}}+\nm{b}{m^{p-1},N}^G\nm{h^{N/2}_G m}{L^{p'}}\Bigr),
\end{equation}
is fulfilled when $b\in S(m^{p-1},G)$, provided that $N$ has been
replaced by a larger number if necessary. In particular we may assume
that
\begin{equation}
h_G^{N/2}m^{p-1}\in L^{p'}(W),\qquad\text{with $N\ge \kappa
_{p'}$}\quad \text{and}\quad S_N(m^{p-1},G)\cap L^{p'}\subseteq
s^w_{p'}. \label{eqn:10}
\end{equation}

\par

Next let $U_j$ and $X_j$   for $j\in \mathbf N$ be the same
as in Remark \ref{remfeasible} after $g$ has been replaced by $G$, and
let $\ep _0$ and $N_0$ be the same as $\ep$ and $N$ respectively in
Remark \ref{remfeasible}. Also let $I_0$ be the set of all $j\in
\mathbf N$ such that $2h_g(X)^{N/2}m(X)\le
|a(X)|$ for some $X\in \overline{U}_j$, and set for each $J\subseteq
\mathbf N$,
$\Omega _J=\cup _{j\in J}U_j$. For each $j\in \mathbf N$ we choose a
point $Y_j\in \overline{U}_j$ such that $|a(X)|\le
|a(Y_j)|$ when $X\in \overline{U}_j$. Then it follows that $I_0$ is an
infinite set and that $\nm a{L^p(\Omega _{I_0})}=+\infty$,
since $a\in L^p_{\mathrm {loc}}(W)\setminus L^p(W)$.

\par

In a moment we shall prove that there are constants $C$ and $r_0>0$,
and a sequence $\{ X_j^0 \} _{j\in I_0}$ such that for any $j\in I_0$
it holds
\begin{gather}
U_j^0=\sets{Y\in W}{G_{X_j}(Y-X^0_j)<4r_0^2} \subset U_j,
\label{eqn:11}
\intertext{and}
\begin{gathered}
h_g(X)^{N/2}m(X)\le |a(Y_j)|/2 \le |a(X)|\le |a(Y_j)|,
\\[1ex]
|a|^G_k(X)\le C |a(X)|,
\end{gathered}\label{formula2.4}
\end{gather}
for all $X\in U_j^0$ and $k\le N$.

\par

Admitting this for a while we may proceed as follows. Let $U_j^1$ be
the open ball with center at $X_j^0$ and radius $r_0$ (with respect to
the metric $G_{X_j}$), and choose a bounded sequence $\{ \fy _j\}
_{j\in I_0}$ in $S(1,G)$ such that $0\le \fy _j \le 1$, $\fy _j\in
C_0^\infty (U_j^0)$ and $\fy _j =1$ in $U_j^1$. Also let $J$ be an
arbitrary finite subset of $I_0$. Then it follows from
\eqref{eqn:11} and \eqref{formula2.4} and
the fact that there is a bound of overlapping $U_j$, that for some
constant $C$ which is independent of $j\in I_0$ and $J$ it holds
$$
|U_j^1| \le |U_j^0|\le |U_j|\le C|U_j^1|,\quad j\in I_0
$$
and
\begin{align}
&\Big ( \sum _{j\in J} |a(Y_j)|^p |U_j^1| \Big )^{1/p} 
\le C\Big (\sum _{j\in J} \nm a{L^p(U_j^1)}^p\Big )^{1/p}
\label{est3.55}
\\[1ex]
&\qquad\le C^2 \nm a{L^p(\Omega _J)}\le C^3\Big ( \sum _{j\in J}
|a(Y_j)|^p|U_j^1|\Big )^{1/p}. \notag
\end{align}

\par

Now we let
$$
b_J(X) = \Big ( \sum _{j\in J}{a(X)}|a(X)|^{p-2}\fy _j(X)\Big
)/\nm a{L^p(\Omega _J)}^{p-1}.
$$
Then it follows from  \eqref{est3.55} that
$$
\sup_J\nm {b_J}{L^{p'}(\Omega _J)}<\infty.
$$ 

\par

In fact, since there is a bound of overlapping $U_j$, we have
\begin{multline*}
\nm {b_J}{L^{p'}} \le \frac 1{\nm {a}{L^p(\Omega _J)}^{p-1}}\Big (\Big (\sum _{j\in J}|a(X)|^{p-1}\fy _j(X)\Big )^{p'}\, dX\Big )^{1/p'}
\\[1ex]
\le C_1\frac 1{\nm {a}{L^p(\Omega _J)}^{p-1}}\Big (\Big (\sum _{j\in J}\int _{U_j}|a(X)|^{p}\, dX\Big )\Big )^{1/p'}
\\[1ex]
\le C_2\frac 1{\nm {a}{L^p(\Omega _J)}^{p-1}} \Big (\int _{\Omega _J} |a(X)|^p\, dX\Big )^{(p-1)/p}{L^p(\Omega _J)}= C_2,
\end{multline*}
for some constants $C_1$ and $C_2$.

\par

Furthermore, by  \eqref{formula2.4} it follows that the set of all
$b_J$ is a bounded subset of $S_N(m^{p-1},G)$. 
Hence \eqref{eqn:2}$'$ and \eqref{eqn:10} give
$$
\sup_J\nm {b_J}{s_{p'}^w}<\infty.
$$
By Proposition \ref{p1.57B} (1) and \eqref{est3.55} it follows now that there are positive constants $C_1$ and $C_2$ which are independent of $J$ such that
\begin{equation*}
\nm a{s_p^w} \ge C_1|(a,b_J)|\ge C_2 \nm a{L^p(\Omega _J)}.
\end{equation*}
By letting $J$ increase to $I_0$ we therefore obtain $\nm a{s_p^w}\ge
C_2 \nm a{L^p(\Omega _{I_0})}=\infty$, which proves the assertion.

\medspace

It remains to prove \eqref{eqn:11} and \eqref{formula2.4}. From
Lemma~\ref{lemma:1} we have that
\begin{align}
&\sup_{X\in \overline U_j}|a|^G_k(X)\le C\Bigl(\sup_{X\in \overline
U_j}|a(X)|+\sup_{Y\in \overline U_j}|a|^G_N(X)\Bigr)
\label{eqn:12}
\\[1ex]
&\quad\le C \Bigl(|a(Y_j)|+\sup_{Y\in \overline
U_j}|a|^G_N(X)\Bigr),\quad\text{for all
$j\in I_0$  and  $k\le N$.} \notag
\end{align}
On the other hand we have
\begin{equation}\label{formula2.2}
|a|^G_{N}(X) = h_g^{N/2}(X)|a|^g_{N}(X) \le Ch_g^{N/2}(X)m(X)
\le C|a(Y_j)|,
\end{equation}
for all $X\in U_j$. From \eqref{eqn:12} and \eqref{formula2.2} we
obtain
\begin{equation}
\sup_{X\in \overline U_j}|a|^G_k(X)\le C |a(Y_j)|,\qquad\text{for all
$j\in I_0$  and  $k\le N$.}  \label{eqn:13}
\end{equation}

\par

Next we consider the Taylor expansion
$$
a(X)=a(Y_j)+\int_0^1a^{(1)}\bigl(Y_j+t(X-Y_j);X-Y_j\bigr)\,dt,
$$
which, together with \eqref{eqn:13}, yields the estimate
$$
|a(Y_j)|\le |a(X)|+ C|a(Y_j)|G_{X_j}(X-Y_j)^{1/2},
$$
for all $X\in U_j$ and $j\in I_0$. But then we can choose $\ep_1>0$
so small that
$$
|a(Y_j)|\le 2 |a(X)|,\quad\text{for $X\in U_j$,
$G_{X_j}(X-Y_j)<\ep_1^2$ and $j\in I_0$.}
$$
Let $\ep_2=\sqrt c$, with $c$  as   in Remark~\ref{remfeasible} and
define
\begin{gather*}
X_j^0=\frac 12\,\frac {\ep_1}{\ep_1+\ep_2}\,X_j+\frac 12\,\frac
{\ep_1+2\ep_2}{\ep_1+\ep_2}\,Y_j
\intertext{and}
r_0=\frac 14\,\frac {\ep_1\ep_2}{\ep_1+\ep_2}.
\end{gather*}
Then it is  easy to check that \eqref{eqn:11} and \eqref{formula2.4}
are satisfied and this completes the proof.
\end{proof}

\par

\subsection{Necessary and sufficient conditions for symbols  in  other
pseudo-differential calculi}\rule{0pt}{0pt}

\medspace

In this subsection we extend the the results of
Subsection~\ref{subsec21} to other  calculi of pseudo-differential
operators, whose definition is a natural generalization of the Weyl
quantization \eqref{t-op}$'$.

\par

As in the definition of the Weyl quantization given in
Subsection~\ref{SympVS}, let $V$ be a real vector space of finite
dimension $n$, $V'$ its dual space, and   $W=V\times V'$ the
symplectic vector space with the symplectic form \eqref{SympForm}.
Let $t\in \mathbf R$ be fixed, and assume that
$a\in \mathscr S(W)$. Then the pseudo-differential operator
$\Op_t (a)$ is defined by the formula \eqref{t-op}
when $f\in \mathscr S(V)$. We recall that the operator $\Op_t(a)$ is
continuous on
$\mathscr S(V)$, and the definition of $\Op_t(a)$ extends to each
$a\in \mathscr S'(W)$, and then $\Op_t(a)$ is a continuous
operator from $\mathscr S(V)$ to $\mathscr S'(V)$. Moreover,
the map $a\mapsto \Op_t(a)$ from $\mathscr S'(W)$ to the set of
linear and continuous operators from $\mathscr S(V)$ to $\mathscr
S'(V)$ is bijective. (See \cite {Ho3}.) 

\par

We note that $a(x,D)=\Op_0(a)$ is the \emph{standard representation} 
(Kohn-Nirenberg representation) and $\Op ^w(a)=\Op_{1/2}(a)$ is the Weyl
quantization. We also recall that if $s,t\in \mathbf R$ and $a,b\in
\mathscr S'(W)$ are arbitrary, then
\begin{equation}\label{stformula}
\Op_s(a)=\Op_t(b)\quad \Longleftrightarrow \quad a(X)=e^{i(s-t)\Phi
(D)}b(X),
\end{equation}
where $\Phi(X)=\scal x \xi$, $X=(x,\xi)\in V\times V'$, and the
right-hand side of \eqref{stformula} is equivalent to
$$
e^{i(s-t)\Phi(X)}\mathscr F b(X)=\mathscr a(X).
$$
(See the introduction for the definition of the Fourier transform $\mathscr F$.) In particular, $e^{it \Phi(D)}$ is a bijective and continuous mapping on $\mathscr S(W)$ which extends
uniquely to bijective and continuous mapping on $\mathscr
S'(W)$, and to a unitary operator on $L^2(W)$.

\par

The extension of the symbolic calculus to  pseudo-differential
operators of the kind \eqref{t-op} requires that the metric $g$ has to
be \emph{splitted} (see \cite{BC}), i.{\,}e. $g$ should satisfy  the
following identity
\begin{equation}\label{gdiag}
g_X(z,\zeta ) = g_X(z,-\zeta ),
\end{equation}
for all $X\in W$,  $z\in V$, and $\zeta\in V'$. (cf. the discussion
after Theorem 18.5.5 and before Theorem 18.5.10 in \cite{Ho3}.),
Observe that \eqref{gdiag} is  equivalent to 
\begin{equation}
g_X(z,\zeta)=g_{1,X}(z)+g_{2,X}(\zeta),  \label{gsplitted}
\end{equation}
where $g_1$ and $g_2$ are positive definite quadratic forms on $V$ and
$V'$ respectively.

\par

The diagonalization of the metric assume a special form when $g$ is
splitted. Recall that the definition of splitted symplectic basis is
given in Remark~\ref{splittedbasis}.

\par

\begin{lemma}\label{G0lemma}
Assume that $g$ is splitted on $W=V\times V'$. Then for all $X\in W$
there exists a splitted symplectic basis
$e_1,\ldots,e_n,\ep_1,\ldots,\ep_n$ such that
$$
g_X(Z)=\sum_{j=1}^n\lambda_j(X)(z_j^2+\zeta_j^2)
$$
for all $Z=(z,\zeta)=\sum_{j=1}^n(z_j e_j+\zeta_j \ep_j)$.
\end{lemma}

\par

\begin{proof}
Since it is well-known that it is possible to diagonalize two
quadratic forms, it follows from \eqref{gsplitted} that there exists a
splitted symplectic basis  $\tilde e_1,\ldots, \tilde e_n,\tilde
\ep_1,\ldots,\tilde \ep_n$ such that
$$
g_X(Z)=\sum_{j=1}^n (\tilde z_j^2+ \mu _j(X)\tilde \zeta_j^2),
$$
where $Z=\sum_{j=1}^n(\tilde z_j \tilde e_j+\tilde \zeta_j \tilde
\ep_j)$, and $\mu_1\ge \cdots\ge \mu_n>0$. Then it suffices to set
$\lambda_j=\mu_j^{1/2}$, $e_j=\mu^{1/4}\tilde e_j$, and
$\ep_j=\mu_j^{-1/4}\tilde \ep_j$, for $j=1,\ldots,n$.
\end{proof}

\par

The following proposition is contained in  Proposition 18.5.10 of
\cite{Ho3}.

\par

\begin{prop}\label{prop:1}
Assume that $g$ is strongly feasible and splitted, and  that
$m$ is $g$-continuous and $(\sigma,g)$-temperate. Also assume that
$t\in\mathbf R$. Then $e^{i\, t\Phi (D)}$ on $\mathscr S'(W)$
restricts to a homeomorphism on $S(m,g)$.
Furthermore, for every integer $N\ge
0$ and $a \in S(m,g)$ it holds
\begin{equation}\label{eq18.5.10}
e^{it\Phi (D)}a -\sum _{k<N}
\bigl(it\Phi(D)\bigr)^ka/k! \in S(h_g^Nm,g).
\end{equation}
\end{prop}

\par

Now we can state the extension of Theorem~\ref{corthm12}:

\par

\begin{thm}\label{corthm13}
Assume that $p\in [1,\infty ]$, $g$ is strongly feasible and splitted,
$m$ is $g$-continuous and
$(\sigma,g)$-temperate, and that $a\in S(m,g)$. Then the following is
true:
\begin{enumerate}
\item if $h_g^{N/2}m\in L^p(W)$ for some $N\ge 0$, then $a\in
s_{t,p}(W)$ if and only if $a\in L^p(W)$;

\vrum

\item if $h_g^{N/2}m\in L^\infty _0(W)$ for some $N\ge 0$, then
$a\in s_{t,\sharp} (W)$ if and only if $a\in L^\infty _0(W)$.
\end{enumerate}
\end{thm}

\par

Theorem \ref{corthm13} is an immediate consequence of
\eqref{stformula}, Theorem \ref{corthm12} and the following
result. 

\par

\begin{prop}\label{prop18.5.10}
Assume that $p\in [1,\infty]$, $g$ is strongly feasible and splitted,
and that $m$ is $g$-continuous, $(\sigma,g)$-temperate and satisfies
$h_g^{N/2}m \in L^p(W)$ ($h_g^{N/2}m \in L^\infty _0(W)$) for some
$N\ge 0$. If $t\in \mathbf R$ is fixed, then $e^{i\, t\Phi (D)}$ on
$\mathscr S'(W)$ restricts to a continuous isomorphism on $S(m,g)\cap
L^p(W)$ ($S(m,g)\cap L^\infty _0(W)$).
\end{prop}

\par

\begin{proof}
We only prove the result for $p<\infty$. The remaining cases follow by
similar arguments and are left for the reader.

\par

We need to prove that  $b=e^{it\Phi (D)}a\in L^p(W)$ whenever $a\in
S(m,g)\cap L^p(W)$. Let $N_0$, $\{ X_j\} _{j\in \mathbf N}$, $\{U_j\}
_{j\in \mathbf N}$ and $G$ be as
in the proof of Proposition \ref{thm2}, and let $\{\fy _j\} _{j\in
\mathbf N}$ be as in Remark \ref{remsetfeasible}. Also let
$\{\psi _j\} _{j\in \mathbf N}$ be a bounded set of non-negative
functions in $S(1,G)$ such that
$\supp \psi _j\subseteq U_j$ and $\psi _j=1$ on $\supp \fy _j$. Then
$G$ is strongly feasible. Since $h_G=h_g^{1/2}$, it follows from
Proposition~\ref{prop:1} that
$$
e^{it\Phi (D)}a -\sum _{k<N}\bigl(it\Phi(D)\bigr)^ka/k! \in
S(h_G^{N}m,G)\subseteq L^p(W).
$$
We therefore need to prove that $\Phi(D)^ka \in L^p(W)$ when $k<N$.

\par

By Lemma~\ref{G0lemma} there exists a splitted symplectic basis
$e_1,\ldots,e_n,\ep_1,\ldots,\ep_n$ such that
$$
G_{X_j}(Z)=\sum_{i=1}^n\lambda_i(X_j)(z_i^2+\zeta_i^2)
$$
for all $Z=(z,\zeta)=\sum_{i=1}^n(z_i e_i+\zeta_i \ep_i)$.
Let $a_j=\fy _ja$ and $j\in \mathbf N$ and
$$
H_j(z_1,\ldots,z_n,\zeta_1,\ldots,\zeta_n)=a_j(X_j+
z_1 e_1+\cdots +z_n e_n + \zeta _1 \ep _1 +\cdots +\zeta _n \ep _n).
$$
By Theorem 4.13 of \cite{Ad} or Lemma A.1 in the appendix there exists
a positive constant $C$ depending only on $N$, $n$ and $p$, and such
that
$$
\nm {\partial _{z}^{\alpha}\partial _{\zeta}^\beta H_j}{L^p}\le 
C\Bigl (\nm
{H_j}{L^p}+\textstyle\sum\limits_{\abp{\gamma+\delta}=N}\nm
{\partial_z^\gamma\partial_\zeta^\delta H_j}{L^p}\Bigr ),
$$
when $|\alpha+\beta |\le N$. 
In particular we obtain
\begin{equation}
\nm {\scal {D_z} {D_\zeta}^k H_j}{L^p}\le 
C\Bigl (\nm
{H_j}{L^p}+\textstyle\sum\limits_{\abp{\gamma+\delta}=N}\nm
{\partial_z^\gamma\partial_\zeta^\delta H_j}{L^p}\Bigr ),
\label{interp}
\end{equation}
for $k<N$.

\par

Since 
\begin{gather*}
\partial_{z_i}H_j(z_1,\ldots,z_n,\zeta_1,\ldots,\zeta_n) =
a^{(1)}(X+Z;e_i)
\intertext{and}
\partial_{\zeta_i}H_j(z_1,\ldots ,z_n,\zeta_1,\ldots
,\zeta_n)=a^{(1)}(X+Z;\ep_i)
\end{gather*}
for $j=1,\ldots,n$, we have
$$
\scal
{D_z}{D_\zeta}^kH_j(z_1,\ldots ,z_n,\zeta_1,\ldots
,\zeta_n)=\Phi(D)^ka_j(X_j+Z)
$$
and 
\begin{align*}
&|\partial_z^\alpha\partial_\zeta^\beta
H_j(z_1,\ldots,z_n,\zeta_1,\ldots,\zeta_n)|
\le C \lambda(X_j)^{\alpha+\beta}m(X_j)
\\[1ex]
&\qquad\le h_G(X_j)^N m(X_j)=h_g(X_j)^{N/2}m(X_j), 
\end{align*}
where $|\alpha+\beta|= N$ and the constant $C$ does not depend on
$(z,\zeta)$ nor on $j\in \mathbf N$.

\par

From \eqref{interp} it follows that
$$
\nm{\Phi(D)^k a_j}{L^p}\le C\Bigl(\nm {a_j}{L^p}+\nm
{h_g^{N/2}m\psi_j}{L^p}\Bigr),
$$
where the constant $C$ does not depend on $j$.

\par

Since there is a bound of overlapping $U_j$ we
therefore obtain
\begin{align*}
\nm {\Phi(D)^ka}{L^p} &\le C_1 \Big (\sum _{j\in \mathbf N}\nm
{\Phi(D)^ka_j}{L^p}^p\Big )^{1/p}
\\[1ex]
&\le
C_2\Big (\sum _{j\in \mathbf N} \Big (\nm {a_j}{L^p}^p+\nm
{h_g^{N/2}m\psi _j}{L^p}^p\Big ) \Big )^{1/p}
\\[1ex]
&\le C_3\big (\nm a{L^p}+\nm {h_g^{N/2}m}{L^p}\big )<\infty ,
\end{align*}
for some constants $C_1$, $C_2$, $C_3$, and the result follows.
\end{proof}

\par

%%%%%%%%%%%%%%%%%%%%%%%%%%%%%%%%%%%%%%%%%%%%%%%%%%%%%%%%%%%%%%%%%%%%%
\section{Further sufficient conditions for symbols  to define
Schatten-von Neumann operators  in the Weyl calculus}\label{sec4}
%%%%%%%%%%%%%%%%%%%%%%%%%%%%%%%%%%%%%%%%%%%%%%%%%%%%%%%%%%%%%%%%%%%%%

\par

In this section  we combine techniques in
\cite{Ho2} with arguments in the proof of Theorem 4.4$'$ in \cite
{To6}. These investigations lead to Theorem \ref{thm3.9} below, where
other sufficient conditions on $N$, $m$ and $g$ comparing to Theorem
4.4$'$ and Proposition 4.5$'$ in \cite {To6} are presented in order for
the embedding
\begin{equation}
S_N(m,g)\cap L^p(W)\subseteq s_p^w(W),\label{eq1}
\end{equation}
should hold. 

Set
\begin{equation}\label{Npconst1}
\kappa _p'=\begin{cases}
[2n(1/p-1/2)]+1,&\text{if  $1\le p<2$,}
\\[1ex]
0, &\text{if $p=2$},
\end{cases} 
\end{equation}
i.{\,}e.  $\kappa _p'=(\kappa _p+1)/2$ when $1\le p<2$ and $\kappa
_p'=\kappa _p=0$ when $p=2$, where the definition of $\kappa _p$ is
given by \eqref{k_p}.

\medspace

The following result, parallel to Proposition
4.5$'${\,}(1) in \cite{To6},  also generalizes Theorem 3.9 in
\cite{Ho2}. Recall
\eqref{varthetadef} for the definition of $\Lambda _g$.

\par

\begin{thm}\label{thm3.9}
Assume that $p\in [1,2]$ and that $N\ge \kappa _p'$. Also assume that
$g$ is slowly varying on $W$, $G=g+g^{{}_{0}}_{\phantom X}${\!\!\!},
and that $m$ is $g$-continuous such that $\Lambda ^{1/p}_Gh_g^{N/2}m
\in L^p(W )$. Then \eqref{eq1} holds.
\end{thm}

\par

We need some preparation for the proof. The first result is
a generalization of the estimate (3.9) in \cite{Ho2}.

\par

\begin{lemma}\label{bernstein}
Assume that  $p,q\in [1,\infty ]$ are such that $p'\le q$ when $p\ge
2$, and that
$$
N\ge \begin{cases}
\bigl[2n(1/p-1/q')\bigr]+1,&\text{\upshape if $p<2$}, \\
0&\text{\upshape if $p\ge 2$},
\end{cases}
$$
Then there is a constant $C$ such that
\begin{equation}\label{uppsk3}
\nm a{s_p^w}\le C\sum _{j=1}^{2n}\nm {D^N _j a}{L^{q}},
\end{equation}
when $a\in C_0^N(W)$ is supported in a ball of radius one.
\end{lemma}

\par

For the proof we recall that for all $1\le p\le \infty $,
$R>0$, and multi-index $\alpha$
such that $|\alpha |\le N$ we have
\begin{equation}\label{compfknest}
\nm a{L^p}\le (2R)^{|\alpha|}\nm {D^\alpha a}{L^p},\quad\text{\upshape
for all $a\in C^N_0(B_R(X))$.}
\end{equation}

\par

\begin{proof}
We may assume that $a$ is supported in a ball with center at
$X=0$. First assume that $1\le p<2$, and let $a\in C_0^N(B(0))$. By
H{\"o}lder's inequality it follows that is no restriction to assume
that $q<2$, which in particular implies that $q<p'$. Let
$\Omega _0=B_2(0)$ and let
\begin{equation*}
\Omega _j = \sets {X\in \rr {2n}}{|X|\ge 1,\ |X_j|> |X|/(4n)},
\end{equation*}
for $j=1,\dots ,2n$, and choose $\{ \fy _j\} _{j=0,\dots ,2n}\subseteq
S^0_{0,0}$ such that
$\supp \fy _j\subseteq \Omega _j$ and $\sum _{j=0}^{2n}\fy _j=1$. Then
it follows from \eqref{sova} that
$$
\nm a{s_p^w}\le C\nm {\mathscr F_\sigma  a}{L^p} \le C\sum
_{j=0}^{2n}\nm {\fy _j\mathscr F_\sigma  a}{L^p}.
$$

\par

We have to estimate $\nm {\fy _j\mathscr F_\sigma  a}{L^p}$ for
$j=0,\dots ,2n$. First assume that $j=0$. By \eqref{compfknest}, and
Haussdorf-Young's and H{\"o}lder's inequalities it
follows that for some constants $C_1$, $C_2$ and  $C_3$ it holds
\begin{align}\label{Fest11}
\nm {\fy _0\mathscr F_\sigma  a}{L^p}&\le \nm {\fy _0}{L^p}\nm
{\mathscr F_\sigma  a}{L^\infty} \le C_1\nm {\fy _0}{L^p}\nm {a}{L^1}
\\[1ex]
&\le C_2\nm {a}{L^q}\le C_3\sum _{j=1}^{2n}\nm {D^N _j
a}{L^{q}}. \notag
\end{align}
The last inequality follows from \eqref{compfknest}.

\par

Next assume that $j\ge 1$. Since $q'>p$, it follows that we may choose
$r\in [1,\infty]$ such that $1/r=1/p-1/q'$. Then $\psi _j(X)\equiv
X_j^{-N}\fy _j(X)$ belongs to $S^0_{0,0}\cap L^{r}(\rr {2n})$. Hence
by integration by parts and H{\"o}lder's inequality it follows that
\begin{align*}
\nm {\fy _j\mathscr F_\sigma  a}{L^p} &= \Big \| \fy _j \int _{\rr
{2n}} a(X)e^{2i\sigma (\cdo ,X)}\, dX \Big \| _{L^p}
\\[1ex]
&= 2^{-N}\Big \| \psi _j \int _{\rr {2n}}(D_j^Na)(X)e^{2i\sigma (\cdo
,X)}\, dX \Big \|
_{L^p}
\\[1ex]
& \le \nm {\psi _j}{L^{r}}\nm {\mathscr F_\sigma
(D_j^Na)}{L^{q'}}.
\end{align*}
Hence the fact that $q'> 2$ and Haussdorf-Young's inequality give
\begin{equation}\label{Fest12}
\nm {\fy _j\mathscr F_\sigma  a}{L^p} \le C\nm
{D_j^Na}{L^q},\quad\text{for $j=1,\dots ,2n$,}
\end{equation}
where $C=\nm {\psi _j}{L^{r}}$ is finite in view of the
assumptions. The assertion now follows in this case by combining
\eqref{Fest11} and \eqref{Fest12}.

\par

Next assume that $p\ge 2$. Then \eqref{sova} and Haussdorf-Young's and
H\"older's inequalities  give
$$
\nm a{s_p^w}\le C_1\nm {\mathscr F_\sigma  a}{L^p}\le C_2\nm
a{L^{p'}}\le C_3\nm {a}{L^q}.
$$
The assertion now follows from \eqref{compfknest} and the proof is
complete.
\end{proof}

\par

Certain parts and ideas of the next result can be found in the proof
of Lemma 3.8 in \cite {Ho2}. We set, as in \cite{To6},
$$
|a|_{W^p_N(\Omega)}=\sum_{|\alpha|=N}\|\partial ^\alpha
a\|_{L^p(\Omega)},\quad \text{and} \quad
\|a\|_{W^p_N(\Omega)}=\sum_{|\alpha|\le N}\|\partial ^\alpha
a\|_{L^p(\Omega)}.
$$

\par

\begin{lemma}\label{lemma3.8A}
Assume that $\Omega \subseteq \rr {2n}$ is open, bounded and
convex, $p\in [1,\infty ]$, $N\ge \kappa _p'$, and that $\fy \in
C_0^N(\Omega )$. Then there exists a positive constant $C$ such that
\begin{equation}\label{estagain1}
\nm {\fy a}{s_p^w}\le C\Big (\sum _{|\alpha |\le N-1}|a^{(\alpha)}(Y)|
+|a|_{W^\infty _N(\Omega )}\Big ),
\end{equation}
for all $Y\in \Omega$ and all $a\in C^N(\Omega)$.

\par

Furthermore, if $q\in [1,\infty]$, and $\Omega _0\subseteq \Omega$ is
open and non-empty, then
\begin{equation}\label{estagain2}
\nm {\fy a}{s_p^w}\le C\bigl(\nm a{L^q(\Omega _0)}
+|a|_{W^\infty_N(\Omega )}\bigr),
\end{equation}
for all $a\in C^N(\Omega)$.
\end{lemma}

\par

\begin{proof}
By choosing a finite numbers of appropriate open balls $B_k\subset
\Omega $ and $\psi _k\in C_0^\infty (B_k)$, $k\in I$ such that
\begin{gather*}
\supp \fy \subseteq \bigcup _{k\in I} B_k ,\intertext{and}
\sum _{k\in I}\psi
_k=1,\qquad\text{on $\supp \fy$},
\end{gather*}
we reduce ourself to the case that $\supp \fy \subseteq B \subseteq
\Omega$, for some open ball $B$.

\par

Let $Y\in \Omega$ be arbitrary. By Taylor expansion it follows that
$a=b+c$, where
\begin{equation*}
b(X) = T_{a,N-1}(X)\equiv \sum _{|\alpha |\le N-1} \frac {a^{(\alpha
)}(Y)}{\alpha !}(X-Y)^\alpha
\end{equation*}
is the Taylor polynomial of $a$ at $Y$ to the order $N-1$, and
\begin{align*}
c(X) &= R_{a,N-1}(X)
\\[1ex]
&\equiv \sum _{|\alpha |=
N} \frac {N}{\alpha!}\int _0^1 (1-t)^{N-1}a^{(\alpha
)}(Y+t(X-Y))(X-Y)^\alpha\, dt
\end{align*}
is the remainder term. The inequality \eqref{estagain1} follows if we
prove that
\begin{align}
&\nm {\fy b}{s_p^w} \le C\sum _{|\alpha |\le N-1}|a^{(\alpha
)}(Y)|,\label{bcest1}
\intertext{and}
&\nm {\fy c}{s_p^w} \le C|a|_{W^\infty_N(\Omega)},\label{bcest2}
\end{align}
for some constant $C$ which is independent of $Y$ and $a$.

\par

First we prove \eqref{bcest1}. By
straight-forward computations we get
\begin{align*}
\nm {\fy b}{s_p^w} &\le C_1\sum _{|\alpha |\le N-1}\nm {a^{(\alpha
)}(Y)(\cdo -Y)^\alpha \fy}{s_p^w}
\\[1ex]
&= C_1\sum _{|\alpha |\le N-1}|a^{(\alpha )}(Y)|\nm {(\cdo -Y)^\alpha
\fy}{s_p^w} \\
& \le C_2 \sum _{|\alpha |\le N-1}|a^{(\alpha )}(Y)|,
\end{align*}
for some constants $C_1$ and $C_2$. This proves the assertion.

\par

Next we prove \eqref{bcest2}. We have that $\partial ^\alpha
c(Y)=0$ when $|\alpha |\le N-1$, and that $\partial ^\alpha
c(X)=\partial ^\alpha a(X)$ when $|\alpha |= N$, since $c=a-T_{a,N-1}$
and $\partial^\alpha_X(T_{a,N-1})=0$ for $\abp{\alpha}=N$.
Hence for any multi-index $\beta$ such that $\abp{\beta}<N$, it
follows that
\begin{align*}
&\partial ^\beta c (X) = R_{c^{(\beta)},N-\abp{\beta}-1}(X)
\\[1ex]
&\;=\sum_{\abp{\gamma}=N-\abp{\beta}}\hspace{-10pt}\frac
{N-\abp{\beta}}{\gamma!}\int_0^1(1-t)^{N-\abp{\beta}-1}
c^{(\beta+\gamma)}(Y+t(X-Y))(X-Y)^{\gamma}\,dt
\\[1ex]
&\;=\sum_{\abp{\gamma}=N-\abp{\beta}}\hspace{-10pt}\frac
{N-\abp{\beta}}{\gamma!}\int_0^1(1-t)^{N-\abp{\beta}-1}
a^{(\beta+\gamma)}(Y+t(X-Y))(X-Y)^{\gamma}\,dt.
\end{align*}
Hence, there is a constant $C$ which is independent of $Y$ such
that
$$
\nm c{W^{\infty}_N(\Omega )} \le C|a|_{W^\infty _N(\Omega )},
$$
which implies that
$$
|\fy c|_{W^\infty _N(\Omega)} \le C|a|_{W^\infty _N(\Omega )}.
$$
An application of Lemma~ \ref{bernstein} with $q=2$ and H\"older
inequality now give
\begin{equation*}
\nm {\fy c}{s_p^w}\le C_1|\fy c|_{W^2_N(\Omega)}
\le C_2 |\fy c|_{W^\infty _N(\Omega)} \le C_3|a|_{W^\infty
_N(\Omega )},
\end{equation*}
which proves \eqref{bcest2}.

\par

It remains to prove \eqref{estagain2}. By H{\"o}lder's inequality, it
suffices to prove the result for $q=1$, since $\Omega _0$ is
bounded. Let $\Omega _1$ be a non-empty open ball such that 
$\overline{\Omega _1}\subseteq \Omega _0$. By applying the
$L^1(\Omega _1)$-norm
with respect to the $Y$-variables in \eqref{estagain1}, and using
Theorem 4.14 of \cite{Ad} or Lemma A.1 in the appendix, we get
\begin{align*}
\nm {\fy a}{s_p^w} &\le C_1\Big ( \nm a{W^1_{N-1}(\Omega _1)} +
|a|_{W^\infty _N(\Omega )}\Big )
\\[1ex]
&\le C_2 \Big ( \nm a{L^1(\Omega _0)} + |a|_{W^1_{N}(\Omega _0)} +
|a|_{W^\infty _N(\Omega )} \Big )
\\[1ex]
&\le C_3 \Big ( \nm a{L^1(\Omega _0)}
+ |a|_{W^\infty _N(\Omega )} \Big ),
\end{align*}
for some constants $C_1,\dots ,C_3$. This proves \eqref{estagain2}.
\end{proof}

\par

In order to generalize Lemma 3.8 in \cite{Ho2}, it is convenient to
use particular classes of modulation spaces, introduced by
Feichtinger in \cite{Fe}. Assume that $\fy \in \mathscr S(\rr n)
\setminus 0$ is fixed and that $p\in [1,\infty ]$. Then the
(classical) modulation space $M^p(\rr n)$ is
the set of all $f\in \mathscr S'(\rr n)$ such that
$$
\nm f{M^p}\equiv \Big ( \iint |\mathscr F(f\, \fy (\cdo -x))(\xi
)|^p\, dxd\xi \Big )^{1/p}
$$
is finite. (With obvious interpretation when $p=\infty$.) Here recall \eqref{fourtrans}
for the definition of the Fourier transform $\mathscr F$. We note that the
definition of $M^p$ is independent of $\fy \in \mathscr S(\rr
n)\setminus 0$ and that different $\fy$ gives rise to equivalent
norms.

\par

The $M^p$ spaces fulfill the usual (complex)
interpolation properties, i.{\,}e.
\begin{equation}\label{interpolMS}
(M^{p_1},M^{p_2})_{[\theta ]} = M^p,\quad
\frac {1-\theta}{p_1}+\frac {\theta}{p_2}=\frac 1p,\quad 1\le
p_1,p_2<\infty ,\  0\le \theta \le 1.
\end{equation}
(Cf. \cite{BeL, Fe, To3}.)

\par

The next result generalizes Lemma 3.8 in \cite{Ho2}. Here it is
convenient to set
\begin{equation}\label{ha}
|a|_{\Omega ,N}(X) = \sup _{Y\in \Omega,\, \abp \alpha =N}|D^\alpha
a(X+Y)| ,
\end{equation}
when $\Omega \subseteq \rr {2n}$, $N\ge 0$ is an integer and $a \in
C^N(\rr {2n})$.

\par

\begin{lemma}\label{lemma3.8}
Assume that $p\in [1,2]$, and that $N\ge \kappa _p'$. Then there is a
constant $C$ such that 
\begin{equation}\label{uppsk3.8}
\nm a{s_p^w}\le C\bigl ( \nm a{L^p} + \nm {|a|_{B(0),N}}{L^p} \bigr),
\end{equation}
for all $a\in C^N(\rr {2n})$.
\end{lemma}

\par

\begin{proof}
Let $\fy \in C^\infty _0(B(0))$ be fixed such that $\int \fy \,
dX=1$, and set $a_X(Y)=\fy (Y-X)a(Y)$. By Lemma \ref{lemma3.8A} it
follows that
\begin{equation}
\nm {a_X}{s_p^w} \le C \bigl( \nm {a_X}{L^p} + |a|_{B(0),N}(X)
\bigr),\label{uppsk3.82A}
\end{equation}
for some constant $C$. We claim that
\begin{equation}
\nm a{s_p^w} \le C\Bigl ( \int _{\rr {2n}}\nm {a_X}{s_p^w}^p\,
dX\Bigr )^{1/p}.\label{uppsk3.82B}
\end{equation}

\par

Admitting this for a while, we obtain
\begin{align*}
\nm a{s_p^w}
&\le
C_1\Big ( \int _{\rr {2n}}\nm {a_X}{s_p^w}^p\,
dX\Big )^{1/p}
\\[1ex]
&\le C_2  \Bigl ( \int _{\rr {2n}}\bigl(\nm {a_X}{L^p} +
|a|_{B(0),N}(X)\bigr)^p \, dX \Bigr )^{1/p}
\\[1ex]
&\le
C_3 \bigl(\nm a{L^p} + \nm {|a|_{B(0),N}}{L^p}\bigr),
\end{align*}
for some constants $C_1$, $C_2$, $C_3$, and the result follows.

\par

It remains to prove \eqref{uppsk3.82B}. First we note that for some
constant $C$ we have
\begin{equation*}
C^{-1}\nm {\mathscr F_\sigma (a\fy (\cdo -X)}{L^p}\le \nm
{a_X}{s_p^w}\le C\nm {\mathscr F_\sigma (a\fy (\cdo -X)}{L^p}
\end{equation*}
in view of Proposition \ref{p1.57C}, since $\fy$ has compact
support. This implies that
\begin{equation}\label{aXnorm1}
C^{-1}\nm a{M^p}\le \Big (\int \nm {a_X}{s_p^w}^p\, dX\Big )^{1/p} \le
C\nm a{M^p}.
\end{equation}
Hence it suffices to prove that
\begin{equation}\label{spMprelation}
\nm a{s_p^w}\le C\nm a{M^p},\quad \text{when}\quad 1\le p\le 2.
\end{equation}

\par

A proof of \eqref{spMprelation} can be found in \cite{GH,To4}. In
order to be self-contained we present an explicit proof here. First
assume that $p=1$. By \eqref{aXnorm1} we have
\begin{equation*}
\nm{a}{s^w_1}=\Bignm{\int _{\rr {2n}} a_X\,dX}{s^w_1}\le \int _{\rr
{2n}}\nm{a_X}{s^w_1}\, dX \le C\nm a{M^1}.
\end{equation*}
This proves the result in this case. Next we consider the case
$p=2$. We have
\begin{align*}
\nm{a}{s^w_2}^2&=(2\pi)^{-n}\nm{a}{L^2}^2=
(2\pi)^{-n}\int _{\rr {2n}}\Bigl(\int _{\rr
{2n}}\abp{a(Y)\varphi(Y-X)}^2\, dY\Bigr) dX
\\[1ex]
&=(2\pi)^{-n}\int _{\rr {2n}} \nm{a_X}{L^2}^2\,dX=(2\pi)^{-n}\int
_{\rr {2n}} \nm{\mathscr F_\sigma (a_X)}{L^2}^2\,dX \le  C\nm a{M^2}^2,
\end{align*} 
for some constant $C$, and the result follows from this case as well.
The inequality \eqref{spMprelation} now follows for general $p\in
[1,2]$ by interpolation, using Theorem 5.1.2 of \cite{BeL},
Proposition \ref{interpolprop} and \eqref{interpolMS}. This proves the
assertion.
\end{proof}

\par

\begin{proof}[Proof of Theorem \ref{thm3.9}.]
Let $\fy _j$ and $U_j$ for $j\in \mathbf N$ be as in
Remark \ref{remfeasible} and Remark \ref{remsetfeasible}, and let $\{
\psi _j\} _{j\in \mathbf N}$ be a bounded sequence in $S(1,g)$ such that
$\psi _j\in C_0^\infty (U_j)$ and $\psi _j=1$ in the support of $\fy
_j$. Also let $a\in S_N(m,g)\cap L^p(W)$, and
set $a_j =\fy _ja$. For each $j\in \mathbf N$, we choose symplectic
coordinates such that $g_j\equiv g_{X_j}$ attains its diagonal
form. Then $g_{j\phantom X}^{{}_{0}} \equiv  g_{X_j}^{{}_{0}}$ and
$G_j \equiv G_{X_j}$ are also given by their diagonal forms, and these
coordinates form an orthonormal basis for $W$ with respect to
$g_{j}^{{}_{0}}$. Also set
$$
K_j = \sets {X}{g_j(X+Y-X_j)\le c,\ g^{{}_{0}}_j(Y)\le 1},
$$
where $c$ is the same as in \eqref{slowly}.
Since $\psi_j$ is equal to $1$ on the support of $a_j$,  Lemma
\ref{lemma3.8} gives
$$
\nm {a_j}{s_p^w}^p\le C(\nm {a_j}{L^p}^p + I_j),
$$
where
\begin{align*}
I_j &\equiv \int _{\rr {2n}}\Bigl(\sup _{\scriptscriptstyle g_j^0
(Y)\le 1} |a_j|^{g_j^{{}_{0}}}_N(X+Y)\Bigr)^p dX
\\[1ex]
&\le
C_1\int _{\rr {2n}}\Bigl (\sup _{\scriptscriptstyle g_j^0 (Y)\le 1}
|a_j|^{g_j}_N(X+Y)h_{g_j}^{N/2}\Bigr )^p dX \\[1ex]
&\le
C_2(h^{N/2}_{g_j}m_j)^p|K_j|
\le C_3(h^{N/2}_{g_j}m_j)^p\Lambda
_{G_j}|U_j| ,
\end{align*}
for some constants $C$, $C_1$, $C_2$ and $C_3$ which are independent
of $j\in \mathbf N$.

\par

Since there is a bound of overlapping $U_j$, and the fact that
$\Lambda _{G_j}$ is $g$-continuous in view of Proposition
\ref{massinv}, it follows that
\begin{gather*}
\sum _{j\in \mathbf N} (m_jh^{N/2}_{g_j})^p\Lambda _{G_j}|U_j| \le
C\nm {\Lambda _{G}^{1/p}h^{N/2}_{g}m}{L^p}^p,
\intertext{and}
\sum _{j\in \mathbf N}\nm {a_j}{L^p}^p\le C\nm a{L^p}^p,
\end{gather*}
for some constant $C$. The result is now a consequence of the
estimate
\begin{equation}\label{agent008}
\nm a{s^w_p}^p\le C\sum _{j\in \mathbf N}\nm {\fy _ja}{s_p^w}^p
\end{equation}
(see Corollary 4.2 in \cite {To6}). The proof is complete.
\end{proof}

\par

%%%%%%%%%%%%%%%%%%%%%%%%%%%%%%%%%%%%%%%%%
\section{Consequences for a 
particular class of symbols}\label{sec3}
%%%%%%%%%%%%%%%%%%%%%%%%%%%%%%%%%%%%%%%%%

\par

In this section we apply the results from the previous sections on
pseudo-differential operators, where the symbols belong to a certain
types of symbol classes  which are defined in a similar way as $S^r_{\rho
,\delta }$ (cf. the introduction).

\par

For each $r,s\in \mathbf R$ and $\rho ,\delta \in \rr {2n}$, we
let $S^{r,s}_{\rho ,\delta}(\rr {2n})$ be the set of all $a\in
C^\infty (\rr {2n})$ such that
$$
\abp{\partial^\alpha_x \partial^\beta_\xi a(x,\xi)}\le 
C_{\alpha,\beta}\eabs {x}^{{s}-\scal {{ {\Pi}} _2\rho}\alpha + \scal {{{\Pi}} _2
\delta}\beta}
\eabs{\xi}^{r -\scal {{{\Pi}} _1\rho}\beta + \scal {{{\Pi}} _1\delta}\alpha },
$$
for some constants $C_{\alpha ,\beta}$ which are independent of $x$
and $\xi$. Here ${{\Pi}} _j\, : \, \rr {2n}\to \rr n$ are the projections
\begin{align*}
&{ {\Pi}} _1(\rho _1,\dots ,\rho _n,\rho _{n+1},\dots \rho _{2n})=(\rho
_1,\dots ,\rho _n)
\intertext{and}
&{{\Pi}} _2(\rho _1,\dots ,\rho _n,\rho _{n+1},\dots \rho _{2n})=(\rho
_{n+1},\dots ,\rho _{2n}).
\end{align*}
We note that if $r, \rho _0,\delta _0\in \mathbf R$,
\begin{equation*}
\rho _j=
\begin{cases} \rho _0\quad 1\le j\le n{\phantom{+1n2}}
\\[1ex]
0\quad n+1\le j\le 2n
\end{cases}
\quad
\text{and}
\quad
\delta _j=
\begin{cases}
\delta _0\quad 1\le j\le n
\\[1ex]
0\quad n+1\le j\le 2n,
\end{cases}
\end{equation*}
then $S^{r,s}_{\rho ,\delta}=S^r_{\rho _0,\delta _0}$.

\par

A simple computation shows that $S(m,g)=S^{r,s}_{\rho ,\delta}$
when
\begin{equation}
g_{x,\xi}(y,\eta)=\sum_{j=1}^n\eabs {x}^{-2\rho _{n+j}}\eabs
{\xi}^{2\delta _j} y_j^2+ \sum_{j=1}^n\eabs {x}^{2\delta _{n+j}}\eabs
{\xi}^{-2\rho_j} \eta _j^2,
\label{metric}
\end{equation}
and
\begin{equation}
m(x,\xi)=\eabs {x}^{s}\eabs{\xi}^{r},
\label{weight}
\end{equation}
Here we recall that $\eabs {x}=\bigl(1+\abp{x}^2\bigr)^{1/2}$.

\par

The proof of the following lemma is omitted, since the result follows
by similar arguments as in Section 18.4 in \cite {Ho3}. Here and in
what follows it is convenient to use the following convention. Assume
that $\mu ,\nu\in \rr n$ and that $r\in \mathbf R$. Then $\nu <\mu$
and $\nu \le \mu$ mean that $\nu_j<\mu_j$ and $\nu_j\le \mu_j$
respectively, for all $j=1,\ldots, n$. Moreover
\begin{alignat*}{5}
r&<\mu ,&\quad r &\le \mu ,&\quad \mu &<r &\quad &\text{and} &\quad \mu &\le r
\intertext{mean that}
r&<\mu _j,&\quad r &\le \mu _j,&\quad \mu _j&<r & \quad  &\text{and} &\quad \mu _j &\le r
\end{alignat*}
respectively, for all $j=1,\ldots,n$.

\par

\begin{lemma}\label{rhodeltafeasible}
Assume that 
$m$ and $g$ are given by \eqref{weight} and \eqref{metric},
respectively. Then $g$ is splitted,
$$
g^\sigma_{x,\xi}(y,\eta)=\sum_{j=1}^n\eabs {x}^{-2\delta _{n+j}}\eabs
{\xi}^{2\rho_j} y_j^2+ \sum_{j=1}^n\eabs {x}^{2\rho _{n+j}}\eabs
{\xi}^{-2\delta _j} \eta _j^2,
$$
and
$$
h_g(x,\xi)=\max_{1\le j\le n}\eabs {x}^{\delta _{n+j} - \rho
_{n+j}}\eabs {\xi}^{\delta _j - \rho _j}.
$$
Moreover,
\begin{enumerate}
\item if $\rho \le 1$ and $0\le \delta$, then $g$ is slowly varying;

\vrum

\item if $0\le \delta \le \rho \le 1$, then $g$ is feasible;

\vrum

\item if $0\le \delta \le \rho \le 1$ and $\delta<1$, then $g$ is
strongly feasible;

\vrum

\item if $\rho \le 1$ and $0\le \delta$, then $m$ is $g$-continuous.
\end{enumerate}
\end{lemma}

\par

The following result now follows from Theorem \ref{corthm13}
and Lemma \ref{rhodeltafeasible}.

\par

\begin{prop}\label{conseq2}
Assume that $t\in \mathbf R$, $p\in [1,\infty )$,
$r,s\in \mathbf R$, $\rho ,\delta \in \rr {2n}$ are such that
$0\le \delta \le \rho \le 1$ and $\delta<1$, and that $a\in
S^{r,s}_{\rho ,\delta }(\rr {2n})$. Then the following is true:
\begin{enumerate}
\item if either $r< -n/p$ or ${{\Pi}} _1\delta <{{\Pi}} _1\rho$, and
either $s< -n/p$ or ${{\Pi}} _2\delta<{{\Pi}} _2\rho$, then $a\in
s_{t,p}$ if and only if $a\in L^p(\rr {2n})$;

\vrum

\item if either $r\le 0$ or ${{\Pi}} _1\delta <{{\Pi}} _1\rho$, and either
$s\le 0$ or ${{\Pi}} _2\delta < {{\Pi}} _2\rho$, then $a\in s_{t,\infty }$ if
and only if $a\in L^\infty  (\rr {2n})$;

\vrum

\item if either $r<0$ or ${{\Pi}} _1\delta <{{\Pi}} _1\rho$, and either
$s<0$ or ${{\Pi}} _2\delta < {{\Pi}} _2\rho$, then $a\in s_{t,\sharp}$ if
and only if $a\in L^\infty_0 (\rr {2n})$.
\end{enumerate}
\end{prop}

\par

Next we focus on the case when $g$ in \eqref{metric} is not
necessarily feasible and illustrate the differences between Theorem
\ref{thm3.9} and Proposition \ref{thm1}{\,}(1). Set
$$
n_p = [2n(1/p-1/2)].
$$

\par

The following result is an immediate consequence of Proposition
\ref{thm1}  and  Lemma \ref{rhodeltafeasible}.

\par

\begin{prop}\label{prop11}
Assume that $1\le p\le 2$, $\rho ,\delta \in \rr {2n}$ are such that $\rho
\le 1$, $0\le \delta$ and $\rho\le \delta$, and that
$a\in S^{r,s}_{\rho ,\delta }$. Also assume that
\begin{equation}\label{first}
\begin{cases}
r<-n-p(n_p+1/2)\max\limits_{1\le j\le n}(\delta_j-\rho_j),
\\[2ex]
s<-n-p(n_p+1/2)\max\limits_{1\le j\le n}(\delta_{n+j}-\rho _{n+j}),
\end{cases}
\end{equation}
and that $a\in L^p(\rr {2n})$. Then $a\in s_p^w$.
\end{prop}

\par

In order to apply Theorem \ref{thm3.9} we need to analyze $\Lambda_G$
(cf.\ \eqref{varthetadef}) with $G=g+g^{\scriptscriptstyle 0}$.
The symplectic transformation
$$
\begin{cases}
y_j=\eabs{x}^{\delta _{n+}j+\rho _{n+j}}\eabs{\xi}^{-\delta_j-\rho_j}z_j,\\
\eta _j=\eabs{x}^{-\delta _{n+j}-\rho _{n+j}}\eabs{\xi}^{\delta_j+\rho_j}\zeta_j,
\end{cases}
$$
with $j=1,\ldots,n$, puts $G$ in diagonal form
$$
G_{x,\xi}(z,\zeta)=\sum_{j=1}^n
\Bigl(\eabs{x}^{\delta _{n+j}-\rho _{n+j}}\eabs{\xi}^{\delta _j-\rho _j}+1\Bigr)
(z_j^2+\zeta_j^2),
$$
so that
\begin{equation}\label{autobus}
\Lambda_G(x,\xi)=\prod _{j=1}^n\Bigl(\eabs {x}^{\delta _{n+j}-\rho
_{n+j}}\eabs{\xi}^{\delta_j-\rho_j}+1\Bigr).
\end{equation}

\par

The following result is
an immediate consequence of Theorem \ref{thm3.9}, Lemma
\ref{rhodeltafeasible} and \eqref{autobus}.

\par

\begin{prop}\label{prop12}
Assume that $1\le p\le 2$, $\rho ,\delta \in \rr {2n}$ are such that  $\rho \le 1$, $0\le \delta$
and $\rho\le \delta$, and that $a\in S^{r,s}_{\rho ,\delta }$. Also assume that
\begin{equation}\label{second}
\begin{cases}
r<-n-p(n_p+1)\max\limits_{1\le j\le
n}(\delta_j-\rho_j)/2-\sum\limits_{1\le j\le n}(\delta_j-\rho_j),
\\[1ex]
s<-n-p(n_p+1)\max\limits_{1\le j\le
n}(\delta _{n+j}-\rho _{n+j})/2-\sum\limits_{1\le j\le n}(\delta _{n+j}-\rho _{n+j}),
\end{cases}
\end{equation}
and that $a\in L^p(\rr {2n})$. Then $a\in s_p^w$.
\end{prop}

\par

\begin{rem}
We note that Propositions \ref{prop11} and \ref{prop12} do not
contain each others. Consequently, Proposition \ref{thm1}{\,}(1) and
Theorem \ref{thm3.9} do not contain each others as well. In fact, as
simple examples show, the conditions imposed on $r$ and $s$ in
\eqref{first} can be stronger, weaker or not comparable with those in
\eqref{second}.
\end{rem}

\par

%%%%%%%%%%%%%%%%%%%%%%%%%%%%%
\section*{Appendix A.}
%%%%%%%%%%%%%%%%%%%%%%%%%%%%%

\par

In this appendix we consider some of the key estimates of $L^p$-type in Sections \ref{sec2} and \ref{sec4} again, and prove that they can be obtained by using techniques as in \cite{BuTo0}.

\par

We start to consider $n$-sectors in $\rr n$. For each $n$-sector $H$ in $\rr n$, there is a  linear and bijective $T=T_H$ on $\rr n$ such that
$$
H=H_T = \sets {T(x)}{x_j\ge 0,\ j=1,\dots ,n}.
$$
We note that different $T$ might give raise to the same sector, and hence, $T_H$ is not unique. On the other hand, if $A_T$ is the matrix for the linear map $T$, then
$$
\Upsilon (H)\equiv |\det (T_H)|
$$
is independent of the choice of $T_H$.

\par

For any set $\Omega \subseteq \rr n$, $\ep \ge 0$ and $n$-sector $H$, we set
$$
\Omega _{H,\ep } \equiv \sets {x+y}{x\in \Omega ,\ y\in H\cap B_\ep
(0)}.
$$

\par

The following lemma is, to some extent, a stressed version of Theorem
4.10 in Chapter 5 in \cite{BS}. 

\par

\renewcommand{\rubrik}{Lemma A.1}

\par

\begin{tom}
Assume that $\Omega \subseteq \rr n$ is a closed convex set, $N\ge 0$
is an integer, $H$ an $n$-sector, $\ep >0$, and that $p\in
[1,\infty ]$. Then there is a constant $C$, depending on $n$, $N$,
$\ep$ and $\Upsilon (H)$ only such that
\begin{equation}\tag*{(A.1)}
\nm {\partial ^\alpha f}{L^p(\Omega )}\le C\Big (\nm f{L^p(\Omega _{H, \ep })}+\sum _{|\beta |=N} \nm {\partial ^\beta f} {L^p(\Omega _{H, \ep })}\Big ),
\end{equation}
when $|\alpha |\le N$ and $f\in C^N(\Omega _{H,\ep })$.
\end{tom}

\par

For the proof we recall some facts on difference
operators and B-splines in Section 5.4 in \cite {BS}.

\par

Let $\chi _{(0,1)}$ be the characteristic function for the interval
$(0,1)$. Then the function $H_j$, for $j\ge 1$, defined inductively on
the real line by
$$
H_1=\chi _{(0,1)},\quad H_{j+1}=H_1*H_j,\quad j\ge 1,
$$
is called the B-spline of order $j$.

\par

For any $h\in \rr
n$, let $\{ T^j_h \} _{j\ge 1}$ be a sequence of operators on $C(\rr n)$
which is inductively defined by
$$
T^1_hf(x)=f(x+h)-f(x),\quad T^{j+1}_h = T^j _h\circ T^1_h
$$
when $f\in C(\rr n)$. An important relation for the B-splines and the
operator $T^j_h$ when $n=1$ is the relation
\begin{equation}\tag*{(A.2)}
T^j_hf(x) = \int f^{(j)}(x+th)h^jH_j(t)\, dt.
\end{equation}

\par

\begin{proof}[Proof of Lemma A.1.] 
We shall mainly follow the proof of Theorem
4.10 in Chapter 5 in \cite{BS}. We may assume that $\ep \le 1$. Furthermore, by making a change of variables, we may assume that
$$
H=\sets {x\in \rr n}{x_j\ge 0,\ j=1,\dots ,n} .
$$

\par

The result is obviously true when $|\alpha |=0$ and $|\alpha |=N$. We
may therefore assume that $0<|\alpha|<N$. First we consider the case
$n=1$. Since the support of $H_j$ is equal to $[0,j]$ and that the
integral of $H_j$ is equal to $1$, the mean-value theorem and (A.2) give
\begin{equation}\tag*{(A.3)}
\begin{aligned}
T_{\ep /N^2}^jf(x) &= \ep ^jN^{-2j}\int f^{(j)}(x+\ep t/N^2)H_j(t)\,
dt
\\[1ex]
&= \ep ^jN^{-2j} f^{(j)}(x+\theta ),
\end{aligned}
\end{equation}
for some $0\le \theta \le \ep /N$. Furthermore,
$$
f^{(j)}(x)=f^{(j)}(x+\theta ) -\int _0^\theta f^{(j+1)}(x+y)\, dy,
$$
and combining this equality with (A.3) gives
$$
f^{(j)}(x) = \ep ^{-j}N^{2j}T^j_{\ep /N^2}f(x) - \int _0^\theta
f^{(j+1)}(x+y)\, dy.
$$
By applying the $L^p(\Omega )$ norm on the latter equality, and using
the fact that
$$
\Big |\int _0^\theta f^{(j+1)}(x+y)\, dy \Big | \le \Big (\int _0^\theta
|f^{(j+1)}(x+y)|^p\, dy \Big )^{1/p},
$$
by H{\"o}lder's inequality, we get
\begin{equation*}
\nm {f^{(j)}}{L^{p}(\Omega )} \le (2N^2/\ep )^j\nm f{L^p
(\Omega _{H, j\ep /N^2})} + \ep \nm {f^{(j+1)}}{L^{p}(\Omega _{H,
\ep /N})}/N.
\end{equation*}
Iteration of this result gives (A.1).

\par

Next assume that $n\ge 1$ is arbitrary. From the first part of the
proof we get
\begin{equation}\tag*{(A.4)}
\nm {\partial ^j_{k}f}{L^p(\Omega )}\le C(\nm
{f}{L^p (\Omega _{H, \ep /n})}+\nm {\partial ^N_kf}{L^p
(\Omega _{H, \ep /n})}).
\end{equation}
For any arbitrary multi-index $\alpha$ we also let $g_k=
\partial _k^{\alpha _k}\cdots \partial _N^{\alpha _N}f$ and
$r_k=r-\sum _{i=k+1}^{n}\alpha _i$. From (A.4) we get
$$
\nm {g_k}{L^{p}(\Omega _{H, (k-1)\ep /n})} \le C (\nm
{g_{k+1}}{L^{p}(\Omega _{H, k\ep /n})} +|f|_{W^p _N(\Omega
_{H, k\ep /n})}).
$$
This gives
\begin{multline*}
\nm {\partial ^\alpha f}{L^p (\Omega )} = \nm {g_1}{L^p
(\Omega )} \le C_1 (\nm {g_{2}}{L^{p}(\Omega _{H, \ep /n})} +|f|_{W^p
_N(\Omega _{H, \ep /n})})\le \cdots
\\[1ex]
\le C_{n-1}(\nm {g_{n}}{L^{p}(\Omega _{H, (n-1)\ep /n})}
+|f|_{W^p _N(\Omega _{H, (n-1)\ep /n})})
\\[1ex]
\le C_n(\nm f{L^p
(\Omega _{H, \ep })} + |f|_{W^p _N(\Omega _{H, \ep})}),
\end{multline*}
for some constants $C_k$ which only depend on $\ep$, $n$ and $N$. The
proof is complete.
\end{proof}

\par

Next we apply Lemma A.1 to a family of subsets of $\rr n$ which contains each convex sets. A subset $\Omega$ of $\rr n$ is called \emph{conistic} (of order $\ep >0$) if for each $x\in \Omega$, there is an $n$-sector $H$ in $\rr n$ such that
\begin{equation}\tag*{(A.5)}
\Upsilon (H)\ge \ep\quad \text{and}\quad x+(H\cap B_\ep (0))\subseteq \Omega .
\end{equation}
By straight-forward computations it follows that any convex set is
conistic. Consequently, since the euclidean structure in Lemma
\ref{lemma:1} is completely determined by the euclidean metric $g_X$
(note here that $X$ is fixed), Lemma \ref{lemma:1} is a consequence of the
following result.

\medspace

\renewcommand{\rubrik}{Proposition A.2}

\par

\begin{tom}
Assume that $\Omega \subseteq
\rr n$ is bounded and conistic of order $\ep >0$, and that $N\in
\mathbf N$. Then there exists a positive constant $C$, depending on
$n$, $N$ and $\ep$ only such that
$$
\nm {\partial ^\alpha f}{L^\infty (\Omega )}\le C\big ( \nm f{L^\infty
(\Omega )} + \sum _{|\beta |=N}\nm {\partial ^\beta f}{L^\infty
(\Omega )} \big ),\quad |\alpha |\le N,\ f\in C^N(\rr n).
$$
\end{tom}

\medspace

\begin{proof}
We may assume that $\Omega$ is a closed set. Let $x_0\in \Omega$  be chosen such that
$$
|\partial ^\alpha f(x_0)| = \nm {\partial ^\alpha f}{L^\infty (\Omega )},
$$
and let the sector $H$ be chosen such that (A.5) is fulfilled for $x=x_0$.
If $\omega =x_0+(H\cap B_{\ep /2}(0))$, then it follows that
$$
\omega _{H,\ep /2} = x_0 + (H\cap B_{\ep }(0))\subseteq \Omega .
$$
Hence Lemma A.1 gives
\begin{multline*}
\nm {\partial ^\alpha f}{L^\infty (\Omega )}\ = |\partial ^\alpha
f(x_0)| = \nm {\partial ^\alpha f}{L^\infty (\omega )}\
\\[1ex]
\le C\Big ( \nm f{L^\infty (\omega _{H,\ep /2})} + \sum _{|\beta
|=N}\nm {\partial ^\beta f}{L^\infty (\omega _{H,\ep /2})} \Big )
\\[1ex]
C\Big ( \nm f{L^\infty (\Omega )} + \sum _{|\beta |=N}\nm {\partial
^\beta f}{L^\infty (\Omega )} \Big ),
\end{multline*}
when $f\in C^N$, and the result follows.
\end{proof}

\end{document}